\documentclass[11pt]{amsart}
\usepackage{latexsym,amscd,amssymb, graphicx, color, amsthm, bm}  
\usepackage{young}
\usepackage[margin=1in]{geometry}
\usepackage{hyperref}
\usepackage[all,cmtip]{xy}

\numberwithin{equation}{section}

\newtheorem{theorem}{Theorem}[section]
\newtheorem{proposition}[theorem]{Proposition}
\newtheorem{corollary}[theorem]{Corollary}
\newtheorem{lemma}[theorem]{Lemma}

\newtheorem{problem}[theorem]{Problem}

\theoremstyle{definition}
\newtheorem{defn}[theorem]{Definition}

\newcommand{\maj}{{\mathrm {maj}}}
\newcommand{\inv}{{\mathrm {inv}}}

\newcommand{\Des}{{\mathrm {Des}}}
\newcommand{\Val}{{\mathrm {Val}}}
\newcommand{\Rise}{{\mathrm {Rise}}}
\newcommand{\Stir}{{\mathrm {Stir}}}
\newcommand{\Hilb}{{\mathrm {Hilb}}}
\newcommand{\grFrob}{{\mathrm {grFrob}}}
\newcommand{\des}{{\mathrm {des}}}

\newcommand{\rev}{{\mathrm {rev}}}

\newcommand{\NSym}{{\mathbf {NSym}}}
\newcommand{\Sym}{{\mathrm {Sym}}}
\newcommand{\QSym}{{\mathrm {QSym}}}
\newcommand{\SYT}{{\mathrm {SYT}}}

\newcommand{\Frob}{{\mathrm {Frob}}}

\newcommand{\initial}{{\mathrm {in}}}
\newcommand{\Ch}{{\mathrm {Ch}}}
\newcommand{\ch}{{\mathbf {ch}}}

\newcommand{\neglex}{{\mathtt {neglex}}}

\newcommand{\iDes}{{\mathrm {iDes}}}
\newcommand{\shape}{{\mathrm {shape}}}

\newcommand{\symm}{{\mathfrak{S}}}

\newcommand{\QQ}{{\mathbb {Q}}}
\newcommand{\ZZ}{{\mathbb {Z}}}

\newcommand{\OP}{{\mathcal{OP}}}
\newcommand{\FF}{{\mathbb{F}}}

\newcommand{\KK}{{\mathbb{K}}}

\newcommand{\CCC}{{\mathcal{C}}}

\newcommand{\AAA}{{\mathcal{A}}}

\newcommand{\BBB}{{\mathcal{B}}}
\newcommand{\GS}{{\mathcal{GS}}}

\newcommand{\zz}{{\mathbf {z}}}
\newcommand{\xx}{{\mathbf {x}}}
\newcommand{\II}{{\mathbf {I}}}

\newcommand{\TT}{{\mathbf {T}}}

\newcommand{\ii}{{\mathbf{i}}}
\newcommand{\sss}{{\mathbf{s}}}
\newcommand{\dd}{{\mathbf {d}}}

\newcommand{\pib}{\overline{\pi}}

\begin{document}

\title[Ordered set partitions and the 0-Hecke algebra]
{Ordered set partitions and the 0-Hecke algebra}

\author{Jia Huang}
\address
{Department of Mathematics \newline \indent
University of Nebraska at Kearney \newline \indent
Kearney, NE, 68849, USA}
\email{huangj2@unk.edu}

\author{Brendon Rhoades}
\address
{Department of Mathematics \newline \indent
University of California, San Diego \newline \indent
La Jolla, CA, 92093, USA}
\email{bprhoades@math.ucsd.edu}

\begin{abstract}
Let the symmetric group $\symm_n$ act on the polynomial ring 
$\QQ[\xx_n] = \QQ[x_1, \dots, x_n]$ by variable permutation.  The coinvariant algebra
is the graded $\symm_n$-module $R_n := {\QQ[\xx_n]} / {I_n}$, where $I_n$ is the ideal in $\QQ[\xx_n]$
generated by invariant polynomials with vanishing constant term.  
Haglund, Rhoades, and 
Shimozono introduced a new quotient $R_{n,k}$ of the polynomial ring $\QQ[\xx_n]$ depending 
on two positive integers $k \leq n$ which reduces to the classical coinvariant algebra of the symmetric group $\symm_n$ when $k = n$.
The quotient $R_{n,k}$ carries the structure of a graded $\symm_n$-module;
 Haglund et. al. determine its graded isomorphism type and relate it to the Delta Conjecture in the theory of Macdonald 
 polynomials.  We introduce and study a related quotient $S_{n,k}$ of $\FF[\xx_n]$ which carries a graded 
 action of the 0-Hecke algebra
 $H_n(0)$, where $\FF$ is an arbitrary field. 
 We prove 0-Hecke analogs of the results of Haglund, Rhoades, and Shimozono.
 In the classical case $k = n$, we recover earlier results of Huang concerning 
 the 0-Hecke action on the coinvariant algebra.
\end{abstract}

\keywords{Hecke algebra, set partition, coinvariant algebra}
\maketitle

\section{Introduction}
\label{Introduction}

The purpose of this paper is to define and study a 0-Hecke analog
of a recently defined graded module for the symmetric group \cite{HRS}.
Our construction has connections with the combinatorics of ordered set partitions
and the Delta Conjecture \cite{HRW} in the theory of Macdonald polynomials.

The symmetric group $\symm_n$ acts on the polynomial ring
$\QQ[\xx_n] := \QQ[x_1, \dots, x_n]$ by variable permutation.
The corresponding {\em invariant subring} $\QQ[\xx_n]^{\symm_n}$ consists of all $f \in \QQ[\xx_n]$ with $w(f) = f$  
for all $w \in \symm_n$, 
and is generated 
by the elementary symmetric functions $e_1(\xx_n), \dots, e_n(\xx_n)$, where
\begin{equation}
e_d(\xx_n) = e_d(x_1, \dots, x_n) := \sum_{1 \leq i_1 < \cdots < i_d \leq n} x_{i_1} \cdots x_{i_d}.
\end{equation}
The {\em invariant ideal} $I_n \subseteq \QQ[\xx_n]$ is the ideal generated by those invariants
$\QQ[\xx_n]^{\symm_n}_+$ with
vanishing constant term:
\begin{equation}
I_n := \langle \QQ[\xx_n]^{\symm_n}_+ \rangle = \langle e_1(\xx_n), \dots, e_n(\xx_n) \rangle.
\end{equation}
The {\em coinvariant algebra} $R_n := {\QQ[\xx_n]} / {I_n}$ is the corresponding quotient ring.

The coinvariant algebra $R_n$ inherits a graded action of $\symm_n$ from $\QQ[\xx_n]$.
This module is among the most important representations in algebraic and geometric combinatorics.
Its algebraic properties are closely tied to the combinatorics of permutations in $\symm_n$;
let us recall some of these properties.

\begin{itemize}
\item The quotient $R_n$ has dimension $n!$ as a  $\QQ$-vector space.  In fact, E. Artin~\cite{Artin} used Galois theory to prove that the set of `sub-staircase' monomials
$\AAA_n := \{ x_1^{i_1} \cdots x_n^{i_n} \,:\, 0 \leq i_j < j \}$ descends to a basis for $R_n$.
\item A different monomial 
basis $\GS_n$ of $R_n$ was discovered by Garsia and Stanton \cite{GS}.  Given a permutation
$w = w(1) \dots w(n) \in \symm_n$, the corresponding GS monomial basis element is
\begin{equation*}
gs_w := \prod_{w(i) > w(i+1)} x_{w(1)} \cdots x_{w(i)}.
\end{equation*}
\item Chevalley \cite{C} proved that $R_n$ is isomorphic as an ungraded $\symm_n$-module to the regular representation
$\QQ[\symm_n]$.
\item Lusztig (unpublished) and Stanley described the {\em graded} 
$\symm_n$-module structure of $R_n$ using the major index statistic on standard Young tableaux \cite{Stanley}.
\end{itemize}

Let $k \leq n$ be two positive integers.  Haglund, Rhoades, and Shimozono \cite[Defn. 1.1]{HRS} introduced
the ideal $I_{n,k} \subseteq \QQ[\xx_n]$ with generators 
\begin{equation}
I_{n,k} := \langle x_1^k, x_2^k, \dots, x_n^k, e_n(\xx_n), e_{n-1}(\xx_n), \dots, e_{n-k+1}(\xx_n) \rangle
\end{equation}
and studied the corresponding quotient ring $R_{n,k} := {\QQ[\xx_n]} / {I_{n,k}}$.  Since $I_{n,k}$ is homogeneous
and
stable 
under the action of $\symm_n$, the ring $R_{n,k}$ is a graded $\symm_n$-module.
When $k = n$, we have $I_{n,n} = I_n$, so that   $R_{n,n} = R_n$ and we recover the usual 
invariant ideal and coinvariant algebra.

To study $R_{n,k}$ one needs the notion of an {\em ordered set partition} of $[n] := \{1, 2, \dots, n\}$, which is a set partition of $[n]$ with a total order on its blocks.
For example, we have an ordered set partition
\begin{equation*}
\sigma = (25 \mid 6 \mid 134)
\end{equation*}
written in the `bar notation'. The three blocks $\{2,5\}$, $\{6\}$, and $\{1,3,4\}$ are ordered from left to right, and elements of each block are increasing.

Let $\OP_{n,k}$ denote the collection of ordered set partitions of $[n]$ with $k$ blocks.  We have
\begin{equation}
|\OP_{n,k}| = k! \cdot \Stir(n,k),
\end{equation}
where $\Stir(n,k)$ is the (signless) Stirling number of the second kind counting $k$-block
set partitions of $[n]$.
The symmetric group $\symm_n$ acts on $\OP_{n,k}$ by permuting the letters $1,\ldots,n$.  For example, the permutation $w=241365$, written in one-line notation, sends $(25 \mid 6 \mid 134)$ to $(46 \mid 5 \mid 123)$.

Just as the  structure of the classical coinvariant module $R_n$ is controlled by permutations
in $\symm_n$, the structure of $R_{n,k}$ is governed by  the collection $\OP_{n,k}$ of
ordered set partitions of $[n]$ with $k$ blocks \cite{HRS}.  
\begin{itemize}
\item  The dimension of $R_{n,k}$ is $|\OP_{n,k}| = k! \cdot \Stir(n,k)$ \cite[Thm. 4.11]{HRS}.
We have a generalization $\AAA_{n,k}$ of the Artin monomial basis to $R_{n,k}$ \cite[Thm. 4.13]{HRS}.
\item There is a generalization
$\GS_{n,k}$ of the Garsia-Stanton monomial basis to $R_{n,k}$ \cite[Thm. 5.3]{HRS}.
\item The module $R_{n,k}$ is isomorphic as an {\em ungraded} $\symm_n$-representation to $\OP_{n,k}$ \cite[Thm. 4.11]{HRS}.
\item There are explicit descriptions of the {\em graded} $\symm_n$-module structure of $R_{n,k}$
which generalize the work of Lusztig--Stanley \cite[Thm 6.11, Cor. 6.12, Cor. 6.13, Thm. 6.14]{HRS}.
\end{itemize}

Now let $\FF$ be an arbitrary field and let $n$ be a positive integer.  
The 
{\em (type A) $0$-Hecke algebra} $H_n(0)$ 
is the unital associative $\FF$-algebra with generators $\pi_1, \pi_2, \dots, \pi_{n-1}$ and relations
\begin{equation}
\label{zero-hecke-relations}
\begin{cases}
\pi_i^2 =  \pi_i & 1 \leq i \leq n-1, \\
\pi_i \pi_j = \pi_j \pi_i & |i - j| > 1,  \\
\pi_i \pi_{i+1} \pi_i = \pi_{i+1} \pi_i \pi_{i+1} & 1 \leq i \leq n-2.
\end{cases}
\end{equation}
Recall that the symmetric group $\symm_n$ has Coxeter generators $\{s_1, s_2, \dots, s_{n-1} \}$, where
$s_i$ is the adjacent transposition $s_i = (i, i+1)$.
These generators satisfy similar relations as \eqref{zero-hecke-relations} except that $s_i^2=1$ for all $i$. 
If $w \in \symm_n$ is a permutation and $w = s_{i_1} \cdots s_{i_{\ell}}$ is a reduced (i.e., as short as possible)
expression for $w$
in the Coxeter generators $\{s_1, \dots, s_{n-1}\}$, we define 
the 0-Hecke algebra element $\pi_w := \pi_{i_1} \cdots \pi_{i_{\ell}} \in H_n(0)$.  It can be shown that the 
set $\{ \pi_w \,:\, w \in \symm_n \}$ forms a basis for $H_n(0)$ as an $\FF$-vector space, and in particular
$H_n(0)$ has dimension $n!$. 
In contrast to the situation with the symmetric group, the representation theory of the 
0-Hecke algebra is insensitive to the choice of ground field, which motivates our generalization
from $\QQ$ to $\FF$.

The algebra $H_n(0)$ is a deformation of the symmetric group algebra
$\FF[\symm_n]$.  Roughly speaking, whereas in a typical $\FF[\symm_n]$-module the 
generator $s_i$ acts by `swapping' the letters $i$ and $i+1$, in a typical 
$H_n(0)$-module the generator $\pi_i$ acts by `sorting' the letters $i$ and $i+1$.
Indeed, the relations satisfied by the $\pi_i$ are precisely the relations satisfied by bubble sorting operators
acting on a length $n$ list of entries $x_1 \dots x_n$  from a totally ordered alphabet:
\begin{equation}
\pi_i .( x_1 \dots x_i x_{i+1} \dots x_n) := \begin{cases}
x_1 \dots x_{i+1} x_i \dots x_n & x_i > x_{i+1} \\
x_1 \dots x_i x_{i+1} \dots x_n & x_i \leq x_{i+1}.
\end{cases}
\end{equation}

Proving 0-Hecke analogs of  module theoretic results concerning the symmetric group has received
a great deal of recent study in algebraic combinatorics \cite{BBSSZ, Huang, HuangTab, VT}; let us 
recall the 0-Hecke analog of the variable permutation action of $\symm_n$ on a polynomial ring.

Let $\FF[\xx_n] := \FF[x_1, \dots, x_n]$ be the polynomial ring in $n$ variables over the field $\FF$.
The algebra $H_n(0)$ acts on $\FF[\xx_n]$ by the {\em isobaric Demazure operators}:
\begin{equation}\label{Demazure}
\pi_i(f) := \frac{x_i f - x_{i+1} (s_i(f))}{x_i - x_{i+1}},\quad 1\le i\le n-1.
\end{equation}
If $f \in \FF[\xx_n]$ is symmetric in the variables $x_i$ and $x_{i+1}$, then $s_i(f) = f$ and thus $\pi_i(f) = f$.
The isobaric Demazure operators give a 0-Hecke analog of variable permutation.

We also have a 0-Hecke analog of the permutation action of $\symm_n$ on $\OP_{n,k}$.
Let $\FF[\OP_{n,k}]$ be the $\FF$-vector space with basis given by $\OP_{n,k}$.
Then $H_n(0)$ acts on $\FF[\OP_{n,k}]$ by the rule
\begin{equation}
\label{pi-defining-action}
\pi_i . \sigma := \begin{cases}
0 & \text{if $i+1$ appears in a block to the left of $i$ in $\sigma$,} \\
\sigma + s_i.\sigma & \text{if $i+1$ appears in a block to the right of $i$ in $\sigma$.} \\
\sigma & \text{if $i+1$ appears in the same block as $i$ in $\sigma$,} 
\end{cases}
\end{equation}
For example, we have
\begin{align*}
\pi_1 (25 \mid 6 \mid 134) &= 0, \\
\pi_2 (25 \mid 6 \mid 134) &= (25 \mid 6 \mid 134) + (35 \mid 6 \mid 124), \\
\pi_3 (25 \mid 6 \mid 134) &= (25 \mid 6 \mid 134).
\end{align*}
It is straightforward to check that these operators satisfy the relations (\ref{zero-hecke-relations})
and so define a 0-Hecke action on $\FF[\OP_{n,k}]$. 
In fact, this is a special case of an $H_n(0)$-action on generalized ribbon tableaux introduced in~\cite{HuangTab}.
See also the proof of Lemma~\ref{osp-substructure}.

The coinvariant algebra $R_n$ can be viewed as a 0-Hecke module.
Indeed, we may express the operator $\pi_i$ as a composition $\pi_i = \partial_i x_i$, where 
$\partial_i$ is the {\em divided difference operator}
\begin{equation}
\partial_i(f) := \frac{f - s_i(f)}{x_i - x_{i+1}}.
\end{equation}
The ``Leibniz rule''
\begin{equation}
\label{leibniz}
\partial_i(f g) = \partial_i(f) g + s_i(f) \partial_i(g)
\end{equation}
has the consequence
\begin{equation}
\label{leibniz-consequence}
\pi_i(fg) = \partial_i(f) (x_i g) + s_i(f) \pi_i(g)
\end{equation}
for all $1 \leq i \leq n-1$ and all $f, g \in \FF[\xx_n]$.
In particular, if $f$ is a symmetric polynomial then $\pi_i(fg) = f \pi_i(g)$ and
 the invariant ideal 
$I_n \subseteq \FF[\xx_n]$ generated by $e_1(\xx_n), \dots, e_n(\xx_n) \in \FF[\xx_n]$ 
is stable under the action of $H_n(0)$ on $\FF[\xx_n]$.  Therefore,
 the quotient $R_n = {\FF[\xx_n]} / {I_n}$ inherits a 0-Hecke action.
Huang gave explicit formulas for its degree-graded and length-degree-bigraded quasisymmetric 0-Hecke
characteristic \cite[Cor. 4.9]{Huang}.
The bivariant characteristic $\Ch_{q,t}(R_n)$ turns out to be a generating function for the 
pair of Mahonian statistics ($\inv$,  $\maj$) on permutations in $\symm_n$,
weighted by the Gessel fundamental quasisymmetric function $F_{\iDes(w)}$
corresponding to the inverse descent set of $\iDes(w)$ of $w\in\symm_n$ \cite[Cor. 4.9 (i)]{Huang}.

We will study a 0-Hecke analog of the rings $R_{n,k}$ of Haglund, Rhoades, and Shimozono \cite{HRS}.
For $k < n$ the ideal $I_{n,k}$ is not usually stable under the action of $H_n(0)$ on 
$\FF[\xx_n]$, so that the quotient ring $R_{n,k} = {\FF[\xx_n]} / {I_{n,k}}$ does not have the structure of 
an $H_n(0)$-module.  To remedy this situation, we introduce the following modified family of ideals.
Let 
\begin{equation}
h_d(x_1, \dots, x_i) := \sum_{1 \leq j_1 \leq \cdots \leq j_d \leq i} x_{j_1} \cdots x_{j_d} 
\end{equation}
be the complete homogeneous symmetric function of degree $d$ in the variables $x_1, x_2, \dots, x_i$.

\begin{defn}
For two positive integers $k \leq n$, we define a quotient ring 
\begin{equation}
S_{n,k} := {\FF[\xx_n]} / {J_{n,k}}
\end{equation}
where $J_{n,k} \subseteq \FF[\xx_n]$ is the ideal with generators
\begin{equation}
J_{n,k} := \langle h_k(x_1), h_k(x_1, x_2), \dots, h_k(x_1, x_2, \dots, x_n), e_n(\xx_n), e_{n-1}(\xx_n), \dots, e_{n-k+1}(\xx_n) 
\rangle.
\end{equation}
\end{defn}

The ideal $J_{n,k}$ is homogeneous.
We claim that $J_{n,k}$ is stable under the action of $H_n(0)$.  Since $e_d(\xx_n) \in \QQ[\xx_n]^{\symm_n}$
and $h_k(x_1, \dots, x_i)$ is symmetric in $x_j$ and $x_{j+1}$ for $j \neq i$, 
thanks to  Equation~\ref{leibniz-consequence} this reduces to the observation that
\begin{equation}\label{shift}
\pi_i(h_k(x_1, \dots, x_i)) = h_k(x_1, \dots, x_i, x_{i+1}).
\end{equation}
Thus the quotient
$S_{n,k}$ has the structure of a graded $H_n(0)$-module.

It can be shown that $J_{n,n} = I_n$, so that $S_{n,n} = R_n$ is the classical coinvariant module.  
At the other extreme, we have $J_{n,1} = \langle x_1, x_2, \dots, x_n \rangle$, so that 
$S_{n,1} \cong \FF$ is the trivial $H_n(0)$-module in degree $0$.

Let us remark on an analogy between the generating sets of $I_{n,k}$ and $J_{n,k}$ which may 
rationalize the more complicated generating set of $J_{n,k}$.  The {\em defining representation} 
of $\symm_n$ on $[n]$ is (of course) given by $s_i(i) = i+1, s_i(i+1) = i,$ and $s_i(j) = j$ otherwise.
The generators of $I_{n,k}$ come in two flavors:
\begin{enumerate}
\item high degree elementary invariants $e_n(\xx_n), e_{n-1}(\xx_n), \dots, e_{n-k+1}(\xx_n)$, and
\item a homogeneous system of parameters $\{x_1^k, x_2^k, \dots, x_n^k\}$ of degree $k$ whose linear span
is stable under the action of $\symm_n$ and isomorphic to the defining representation.
\end{enumerate}
\begin{equation*}
\xymatrix @C=5pc{
1 \ar@{<->}[r]^{s_1} &2 \ar@{<->}[r]^{s_2}  &\cdots \ar@{<->}[r]^{s_{n-1}} &n  \\
x_1^k \ar@{<->}[r]^{s_1} &x_2^k \ar@{<->}[r]^{s_2} &\cdots \ar@{<->}[r]^{s_{n-1}} &x_n^k}
\end{equation*}

The {\em defining representation} of $H_n(0)$ on $[n]$ is given by
$\pi_i(i) = i+1$ and $\pi_i(j) = j$ otherwise (whereas $s_i$ acts by {\em swapping} at $i$,
$\pi_i$ acts by {\em shifting} at $i$).
The generators of $J_{n,k}$ come in two analogous flavors:
\begin{enumerate}
\item high degree elementary invariants $e_n(\xx_n), e_{n-1}(\xx_n), \dots, e_{n-k+1}(\xx_n)$, and
\item a homogeneous system of parameters $\{h_k(x_1), h_k(x_1, x_2), \dots, h_k(x_1, x_2, \dots, x_n)\}$ 
of degree $k$ whose linear span
is stable under the action of $H_n(0)$ and isomorphic to the defining representation (see \eqref{shift}).
\end{enumerate}
\begin{equation*}
\xymatrixcolsep{3pc}
\xymatrix{
1 \ar[r]^-{\pi_1} &2 \ar[r]^-{\pi_2}  &\cdots \ar[r]^-{\pi_{n-1}} &n  \\
h_k(x_1) \ar[r]^-{\pi_1} &h_k(x_1, x_2) \ar[r]^-{\pi_2} &\cdots \ar[r]^-{\pi_{n-1}} &h_k(x_1, \dots, x_n)}
\end{equation*}

Deferring various definitions to Section~\ref{Background}, let us state
our main results on $S_{n,k}$.

\begin{itemize}
\item  The module $S_{n,k}$ has dimension $|\OP_{n,k}| = k! \cdot \Stir(n,k)$ as an $\FF$-vector space
(Theorem~\ref{hilbert-series}).
There is a basis $\CCC_{n,k}$ for $S_{n,k}$, generalizing the Artin monomial basis of $R_n$.
(Theorem~\ref{standard-monomial-basis}, Corollary~\ref{staircase-standard}).
\item  There is a generalization $\GS_{n,k}$ of the 
 the Garsia-Stanton monomial basis to $S_{n,k}$ (Corollary~\ref{garsia-stanton-basis}).
\item  As an {\em ungraded} $H_n(0)$-module, the quotient $S_{n,k}$ is isomorphic to $\FF[\OP_{n,k}]$ (Theorem~\ref{s-decomposition-theorem}).
\item  As a {\em graded} $H_n(0)$-module, we have explicit formulas for the degree-graded characteristics $\Ch_t(S_{n,k})$ and $\ch_t(S_{n,k})$ and the length-degree-bigraded characteristic
$\Ch_{q,t}(S_{n,k})$ of 
$S_{n,k}$  (Theorem~\ref{s-characteristic-theorem}, Corollary~\ref{s-characteristic-corollary}).
The degree-graded quasisymmetric characteristic $\Ch_t(S_{n,k})$ is symmetric and
 coincides with the graded Frobenius character 
of the $\symm_n$-module $R_{n,k}$ (over $\QQ$).
\end{itemize}

The remainder of the paper is structured as follows.
In {\bf Section~\ref{Background}} we give background and definitions related to compositions,
ordered set partitions, Gr\"obner theory, and the representation theory of 0-Hecke algebras.
In {\bf Section~\ref{Hilbert}} we will prove that the quotient $S_{n,k}$ has dimension $|\OP_{n,k}|$
as an $\FF$-vector space.  We will derive a formula for the Hilbert series of $S_{n,k}$ and 
give a generalization of the Artin monomial basis to $S_{n,k}$.
In {\bf Section~\ref{Garsia}} we will introduce a  family of bases of $S_{n,k}$ which are related
to the classical Garsia-Stanton basis in a unitriangular way when $k = n$.  
In {\bf Section~\ref{Module}} we will use one particular basis from this family to prove that 
the ungraded 0-Hecke structure of $S_{n,k}$ coincides with $\FF[\OP_{n,k}]$.
In {\bf Section~\ref{Characteristic}}
we  derive formulas for the degree-graded quasisymmetric and noncommutative symmetric characteristics
$\Ch_t(S_{n,k})$ and $\ch_t(S_{n,k})$, and the length-degree-bigraded quasisymmetric characteristics $\Ch_{q,t}(S_{n,k})$ of $S_{n,k}$.
In {\bf Section~\ref{Conclusion}} we make closing remarks.

\section{Background}
\label{Background}

\subsection{Compositions}
Let $n$ be a nonnegative integer.
A {\em (strong) composition} $\alpha$ of $n$ is a sequence
$\alpha = (\alpha_1, \dots, \alpha_{\ell})$ of positive integers with $\alpha_1 + \cdots + \alpha_{\ell} = n$.
We call $\alpha_1,\ldots,\alpha_\ell$ the {\em parts} of $\alpha$.
We write $\alpha \models n$ to mean that $\alpha$ is a  composition of $n$.
We also write $|\alpha| = n$ for the {\em size} of $\alpha$ and $\ell(\alpha) = \ell$ for the number  of parts of $\alpha$.

The {\em descent set} $\Des(\alpha)$ of a composition $\alpha = (\alpha_1, \dots, \alpha_{\ell}) \models n$ is the subset of $[n-1]$ given by
\begin{equation}
\Des(\alpha) := \{\alpha_1, \alpha_1 + \alpha_2, \dots, \alpha_1 + \alpha_2 + \cdots + \alpha_{\ell - 1} \}.
\end{equation}
The map $\alpha \mapsto \Des(\alpha)$ gives a bijection from the set of compositions of $n$ to the collection  
of subsets of $[n-1]$.
The {\em major index} of $\alpha = (\alpha_1, \dots, \alpha_{\ell})$
is 
\begin{equation}
\maj(\alpha) := \sum_{i\in\Des(\alpha)} i = (\ell - 1) \cdot \alpha_1 + \cdots + 1 \cdot \alpha_{\ell-1} + 0 \cdot \alpha_{\ell}.
\end{equation}
Given two compositions $\alpha, \beta \models n$, we write $\alpha \preceq \beta$ if 
$\Des(\alpha) \subseteq \Des(\beta)$.  Equivalently, we have $\alpha \preceq \beta$ if the composition $\alpha$
can be formed by merging adjacent parts of the composition $\beta$.  
If $\alpha \models n$, the {\em complement} $\alpha^c \models n$ of $\alpha$ is the unique composition of $n$
which satisfies $\Des(\alpha^c) = [n-1] \setminus \Des(\alpha)$.

As an example of these concepts, let $\alpha = (2,3,1,2) \models 8$.  We have $\ell(\alpha) = 4$. 
The descent set of $\alpha$ is $\Des(\alpha) = \{2,5,6\}$.
The major index is $\maj(\alpha) = 2+5+6 = 3 \cdot 2 + 2 \cdot 3 + 1 \cdot 1 + 0 \cdot 2 = 13$.
The complement of $\alpha$ is $\alpha^c = (1,2,1,3,1) \models 8$ with descent set $\Des(\alpha^c) = \{1,3,4,7\} = [7]\setminus\{2,5,6\}$.

If $\ii = (i_1, \dots, i_n)$ is any sequence of  integers, the {\em descent set}
$\Des(\ii)$ is given by
\begin{equation}
\Des(\ii) := \{1 \leq j \leq n-1 \,:\, i_j > i_{j+1} \}.
\end{equation}
The {\em descent number} of $\ii$ is $\des(\ii) := |\Des(\ii)|$ and the {\em major index} of $\ii$ is $\maj(\ii) := \sum_{j \in \Des(\ii)} j$.
Finally, the {\em inversion number} $\inv(\ii)$ 
\begin{equation}
\inv(\ii) := | \{ (j,j') : 1\le j<j'\le n,\ i_j > i_{j'} \}
\end{equation}
counts the number of inversion pairs in the sequence $\ii$.

If a permutation $w \in \symm_n$ has one-line notation $w = w(1) \cdots w(n)$, we define 
$\Des(w)$, $\maj(w)$, $\des(w)$, and $\inv(w)$ as in the previous paragraph for the sequence $(w(1),\ldots,w(n))$.
It turns out that $\inv(w)$ is equal to the {\em Coxeter length} $\ell(w)$ of $w$, i.e., the length of
a \emph{reduced} (shortest possible) expression for $w$ in the generating set $\{s_1, \dots, s_{n-1}\}$ of 
$\symm_n$, where $s_i$ is the adjacent transposition $(i,i+1)$.  
Moreover, we have $i\in\Des(w)$ if and only if some reduced expression of $w$ ends with $s_i$.
We also let $\iDes(w) := \Des(w^{-1})$ be the descent set of the inverse of the permutation $w$.

The statistics $\maj$ and $\inv$ are equidistributed on $\symm_n$ and their common distribution has
a nice form.    Let us recall the standard $q$-analogs of numbers, factorials, and multinomial coefficients:
\begin{align*}
[n]_q := 1 + q + \cdots + q^{n-1} &   &[n]!_q := [n]_q [n-1]_q \cdots [1]_q  \\
{n \brack a_1, \dots , a_r}_q := \frac{[n]!_q}{[a_1]!_q \cdots [a_r]!_q} 
& &{n \brack a}_q := \frac{[n]!_q}{[a]!_q [n-a]!_q}.
\end{align*}
MacMahon \cite{MacMahon} proved  
\begin{equation}
\sum_{w \in \symm_n} q^{\inv(w)} = \sum_{w \in \symm_n} q^{\maj(w)} = [n]!_q,
\end{equation}
and any statistic on $\symm_n$ which shares this distribution is called {\em Mahonian}.
The joint distribution $\sum_{w \in \symm_n} q^{\inv(w)} t^{\maj(w)}$ of the pair of statistics
$(\inv, \maj)$ is called the {\em biMahonian distribution}.

If $\alpha \models n$ and $\ii = (i_1, \dots, i_n)$ 
is a sequence of  integers of length $n$, we define $\alpha \cup \ii \models n$ 
to be the unique composition of $n$ which satisfies
\begin{equation}
\Des(\alpha \cup \ii) = \Des(\alpha) \cup \Des(\ii).
\end{equation}
For example, let $\alpha = (3,2,3) \models 8$ and let
$\ii = (4,5,0,0,1,0,2,2)$.  We have 
\begin{equation*}
\Des(\alpha \cup \ii) = \Des(\alpha) \cup \Des(\ii) = \{3,5\} \cup \{2,5\} = \{2,3,5\},
\end{equation*}
so that $\alpha \cup \ii = (2,1,2,3)$.
Whenever $\alpha \models n$ and $\ii$ is a length $n$ sequence, we have the 
relation $\alpha \preceq \alpha \cup \ii$.

A {\em partition} $\lambda$ of $n$ is a weakly decreasing sequence $\lambda = (\lambda_1 \geq \cdots \geq \lambda_{\ell})$
of positive integers
which satisfies $\lambda_1 + \cdots + \lambda_{\ell} = n$.  We write
$\lambda \vdash n$ to mean that $\lambda$ is a partition of $n$.  We also write
$|\lambda| = n$ for the {\em size} of $\lambda$ and $\ell(\lambda) = \ell$ for the number of {\em parts}
of $\lambda$.
The {\em (English) Ferrers diagram}
of $\lambda$ consists of $\lambda_i$ left justified boxes in row $i$.

Identifying partitions with Ferrers diagrams, if $\mu \subseteq \lambda$ are a pair of partitions related by 
containment, the {\em skew partition} $\lambda/\mu$ is obtained by removing $\mu$ from $\lambda$.
We write $|\lambda/\mu| := |\lambda| - |\mu|$ for the number of boxes in this skew diagram.
For example, the Ferrers diagrams of $\lambda$ and $\lambda/\mu$ are shown below, where
$\lambda = (4,4,2)$ and $\mu = (2,1)$.

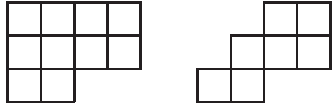
\begin{center}
\begin{small}
\begin{Young}
  & & &  \cr
  & & & \cr
  &
\end{Young}  \hspace{0.2in}
\begin{Young}
,&, & &  \cr
,& & & \cr
&
\end{Young}
\end{small}  
\end{center}

A {\em semistandard tableau} of a skew shape $\lambda/\mu$ is a filling of the Ferrers diagram of $\lambda/\mu$
with positive integers which are weakly increasing across rows and strictly increasing down columns.
A {\em standard tableau} of shape $\lambda/\mu$ is a bijective filling of the Ferrers diagram 
of $\lambda/\mu$ with the numbers $1, 2, \dots, |\lambda/\mu|$ which is semistandard.
An example of a semistandard tableau and a standard tableau
of shape $(4,4,2)/(2,1)$ are shown below.

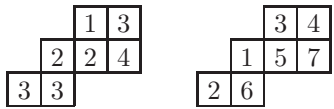
\begin{center}
\begin{small}
\begin{Young}
 , &, & 1 & 3  \cr
 , & 2 & 2 & 4 \cr
  3& 3
\end{Young}  \hspace{0.2in}
\begin{Young}
,&, & 3 & 4  \cr
,& 1 &5 & 7 \cr
2 & 6
\end{Young}
\end{small}
\end{center}

A {\em ribbon} is an edgewise connected skew diagram which contains no $2 \times 2$ square.
The set of compositions of $n$ is in bijective correspondence with the set of size $n$ ribbons: a composition
$\alpha = (\alpha_1, \dots, \alpha_{\ell})$ corresponds to the ribbon whose $i^{th}$ row from the bottom
contains $\alpha_i$ boxes.  We will identify compositions with ribbons in this way.  For example, the 
ribbon corresponding to $\alpha = (2,3,1)$ is shown on the left below.
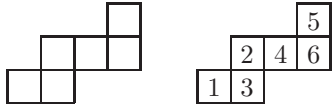
\begin{center}
\begin{small}
\begin{Young}
,&, &, &   \cr
,& & & \cr
&
\end{Young}  \hspace{0.2in}
\begin{Young}
,&, &, & 5   \cr
,& 2 & 4 & 6 \cr
1 & 3
\end{Young}
\end{small}  
\end{center}

Let $\alpha \models n$ be a composition.  We define a permutation $w_0(\alpha) \in \symm_n$ as follows.
Starting at the leftmost column and working towards the right, and moving from top to bottom within each column, fill the 
ribbon diagram of $\alpha$ with the numbers $1, 2, \dots, n$ (giving a standard tableau).
The permutation $w_0(\alpha)$ has one-line notation obtained by reading along the ribbon from the bottom row to the top row, proceeding from left to right within each row.  
It can be shown that $w_0(\alpha)$ is the unique left weak Bruhat minimal permutation $w \in \symm_n$ which
satisfies $\Des(w) = \Des(\alpha)$ (cf. Bj\"orner and Wachs~\cite{BjornerWachs}).
For example, if $\alpha = (2,3,1)$, the figure on the above right
shows $w_0(\alpha) = 132465 \in \symm_6$.

\subsection{Ordered set partitions}
As explained in Section~\ref{Introduction},
an {\em ordered set partition} $\sigma$ of size $n$ is a set partition of $[n]$ with a total order on its blocks.
Let $\OP_{n,k}$ denote the collection of ordered set 
partitions of size $n$ with $k$ blocks.  
In particular, we may identify $\OP_{n,n}$ with $\symm_n$.

Also as in Section~\ref{Introduction},
we write an ordered set partition of $[n]$ as a permutation of $[n]$ with bars to separate blocks, such that letters within each block are increasing and blocks are ordered from left to right.
For example, we have
\begin{equation*}
\sigma = (245 \mid 6 \mid 13) \in \OP_{6,3}.
\end{equation*}
The {\em shape} of an ordered set partition $\sigma = (B_1 \mid \cdots \mid B_k)$ is
the composition $\alpha = (|B_1|, \dots, |B_k|)$.  For example, the above ordered set partition has shape $(3,1,2) \models 6$.

If $\alpha \models n$ is a composition, let $\OP_{\alpha}$ denote the collection of ordered set partitions of $n$ with shape $\alpha$.  
Given an ordered set partition $\sigma \in \OP_{\alpha}$, we can also represent $\sigma$ as the pair
$(w, \alpha)$, where $w = w(1) \cdots w(n)$ is the permutation in $\symm_n$ (in one-line notation) obtained by erasing 
the bars in $\sigma$.  For example, the above ordered set partition becomes
\begin{equation*}
\sigma = (245613, (3,1,2)).
\end{equation*}
This notation establishes a bijection between $\OP_{n,k}$ and pairs $(w, \alpha)$ where $\alpha \models n$ 
is a composition with $\ell(\alpha) = k$ and $w \in \symm_n$ is a permutation with $\Des(w) \subseteq \Des(\alpha)$.


We extend the statistic $\maj$ from permutations to ordered set partitions as follows.
Let $\sigma = (B_1 \mid \cdots \mid B_k) \in \OP_{n,k}$ be an ordered set partition represented as a pair $(w, \alpha)$ as above.  
We define the {\em major index} $\maj(\sigma)$ to be the statistic
\begin{equation}
\maj(\sigma) = \maj(w, \alpha) := \maj(w) + \sum_{i \,  : \,  \max(B_i) < \min(B_{i+1})} (\alpha_1 + \cdots + \alpha_i - i).
\end{equation}
For example, if $\sigma = ( 2 4 \mid 5 7 \mid 1 3 6 \mid 8)$, then
\begin{equation*}
\maj(\sigma) = \maj(24571368) + (2 - 1) + (2 + 2 + 3 - 3) = 4 + 1 + 4 = 9.
\end{equation*}

We caution the reader that our definition of
$\maj$ is {\bf not} equivalent to, or even equidistributed with,
the corresponding statistics for ordered set partitions in \cite{RW, HRS} and elsewhere.  
However,
the distribution of our $\maj$ on $\OP_{n,k}$ is the reversal of the distribution of their $\maj$.

The generating function for $\maj$ on $\OP_{n,k}$ may be described as follows.  
Let $\rev_q$ be the operator on polynomials in the variable $q$ which reverses coefficient sequences.
For example, we have
\begin{equation*}
\rev_q(3q^3 + 2q^2 + 1) = q^3 + 2q + 3.
\end{equation*}
The {\em $q$-Stirling number} $\Stir_q(n,k)$ is defined by the recursion
\begin{equation}
\Stir_q(n,k) = \Stir_q(n-1,k-1) + [k]_q \cdot \Stir_q(n-1,k)
\end{equation}
and the initial condition $\Stir_q(0,k) = \delta_{0,k}$, where $\delta$ is the Kronecker delta.

\begin{proposition}
Let $k \leq n$ be positive integers.  We have
\begin{equation}
\sum_{\sigma \in \OP_{n,k}} q^{\maj(\sigma)} = \rev_q( [k]!_q \cdot \Stir_q(n,k)).
\end{equation}
\end{proposition}

\begin{proof}
To see why this equation holds, consider the statistic $\maj'$ on an ordered set partition 
$\sigma = (B_1 \mid \cdots \mid B_k) = (w, \alpha) \in \OP_{n,k}$ defined by
\begin{equation}
\maj'(\sigma) = \maj'(w,\alpha) := \sum_{i = 1}^k (i-1)(\alpha_i - 1) + \sum_{i \, : \, \min(B_i) > \max(B_{i+1})} i.
\end{equation}
This is precisely the version of major index on ordered set partitions studied
by Remmel and Wilson \cite{RW}.  They proved \cite[Eqn. 15, Prop. 5.1.1]{RW}
that
\begin{equation}
\label{remmel-wilson}
\sum_{\sigma \in \OP_{n,k}} q^{\maj'(\sigma)} = [k]!_q \cdot \Stir_q(n,k).
\end{equation}

On the other hand, for any $\sigma = (w,\alpha) = (B_1 \mid \cdots \mid B_k) \in \OP_{n,k}$ we have
\begin{equation}
\maj(w) = \sum_{i \, : \, \max(B_i) > \min(B_{i+1})} (\alpha_1 + \cdots + \alpha_i).
\end{equation}
This implies
\begin{align}
\maj(\sigma) &= \maj(w) + \sum_{i \, : \, \max(B_i) < \min(B_{i+1})} (\alpha_1 + \cdots + \alpha_i - i) \\
&= \sum_{i = 1}^{k-1} [(k-i) \cdot \alpha_i]   - \sum_{i \, : \, \max(B_i) < \min(B_{i+1})} i.
\end{align}

The longest element $w_0 =  n \dots 21$ (in one-line notation) of $\symm_n$ gives an involution on $\OP_\alpha$ by
\begin{equation*}
\sigma = (B_1 \mid \cdots \mid B_k) \mapsto w_0(\sigma) = (w_0(B_1) \mid \cdots \mid w_0(B_k)).
\end{equation*}

If $\alpha \models n$ and $\ell(\alpha) = k$, then for $\sigma = (B_1 \mid \cdots \mid B_k) \in \OP_{\alpha}$
and any index $1 \leq i \leq k-1$ 
we have $\max(B_i) < \min(B_{i+1})$ if and only if $\min(w_0(B_i)) > \max(w_0(B_{i+1}))$.
Therefore,
\begin{align}
\maj'(\sigma) + \maj(w_0(\sigma)) &= (k-1)(\alpha_k - 1) + \sum_{i = 1}^{k-1} [(i-1)(\alpha_i - 1) + (k-i) \cdot \alpha_i] \\
&= (k-1)(\alpha_k - 1) + \sum_{i = 1}^{k-1} [-\alpha_i - i + 1 + k \alpha_i] \\
&= (k-1)(n - k) - {k \choose 2}.
\end{align}
On the other hand, it is easy to see that 
\begin{equation}
\max\{ \maj(\sigma) \,:\, \sigma \in \OP_{n,k} \} = (k-1)(n - k) - {k \choose 2} =
\max\{ \maj'(\sigma) \,:\, \sigma \in \OP_{n,k} \}.
\end{equation}
Applying Equation~\ref{remmel-wilson} gives
\begin{equation}
\sum_{\sigma \in \OP_{n,k}} q^{\maj(\sigma)} = \rev_q \left[ \sum_{\sigma \in \OP_{n,k}} q^{\maj'(\sigma)} \right]
= \rev_q( [k]!_q \cdot \Stir_q(n,k) ).
\end{equation}
\end{proof}

Now we define an action of the $0$-Hecke algebra $H_n(0)$ on $\FF[\OP_{n,k}]$.
We will find it convenient to introduce an alternative generating set of $H_n(0)$.
For $1 \leq i \leq n-1$, define the element
$\pib_i \in H_n(0)$ by $\pib_i := \pi_i - 1$. 
We will often use the relation $\pib_i\pi_i = \pi_i\pib_i=0$.
It can be shown that the set 
$\{ \pib_1, \pib_2, \dots, \pib_{n-1} \}$ is a generating set of $H_n(0)$ subject
only to the relations
\begin{equation}
\begin{cases}
\pib_i^2 = - \pib_i & 1 \leq i \leq n-1 \\
\pib_i \pib_j = \pib_j \pib_i & |i - j| > 1 \\
\pib_i \pib_{i+1} \pib_i = \pib_{i+1} \pib_i \pib_{i+1} & 1 \leq i \leq n-2.
\end{cases}
\end{equation}
In terms of the $\pib_i$, the 0-Hecke action on $\FF[\OP_{n,k}]$ is given by 
\begin{equation}
\pib_i.\sigma = 
\begin{cases}
-\sigma & \text{if $i+1$ appears in a block to the left of $i$ in $\sigma$,} \\
0 & \text{if $i+1$ appears in the same block as $i$ is $\sigma$, and} \\
s_i(\sigma) & \text{if $i+1$ appears in a block to the right of $i$ in $\sigma$.}
\end{cases}
\end{equation}
This defines an $H_n(0)$-action which preserves $\FF[\OP_\alpha]$ for each composition $\alpha$ of $n$; 
see the proof of Proposition~\ref{osp-substructure}.

\subsection{Gr\"obner theory}
We review material related to Gr\"obner bases of ideals $I \subseteq \FF[\xx_n]$ and standard monomial
bases of the corresponding quotients $\FF[\xx_n]/I$.  For a more leisurely introduction to this material,
see \cite{CLO}.

A total order $\leq$ on the monomials in $\FF[\xx_n]$ is called a {\em term order} if 
\begin{itemize}
\item $1 \leq m$ for all monomials $m \in \FF[\xx_n]$, and
\item  if $m, m', m'' \in \FF[\xx_n]$ are monomials, then $m \leq m'$ implies $m \cdot m'' \leq m' \cdot m''$.
\end{itemize}
In this paper, we will consider the lexicographic term order {\bf with respect to the variable ordering
$x_n > \cdots > x_2 > x_1$}.  That is, we have
\begin{equation*}
x_1^{a_1} \cdots x_n^{a_n} < x_1^{b_1}  \cdots x_n^{b_n}
\end{equation*}
if and only if there exists an integer $1 \leq  j \leq n$ such that $a_{j+1} = b_{j+1}, \dots, a_n = b_n$, and
$a_j < b_j$.
Following the notation of $\textsc{sage}$,
we call this term order {\tt neglex}.

Let $\leq$ be any term order on monomials in $\FF[\xx_n]$.  If $f \in \FF[\xx_n]$ is a nonzero polynomial,
let $\initial_<(f)$ be the leading (i.e., largest) term of $f$ with respect to $<$.  If $I \subseteq \FF[\xx_n]$ is an ideal, the 
associated {\em initial ideal} is the monomial ideal
\begin{equation}
\initial_<(I) := \langle \initial_<(f) \,:\, f \in I - \{0\} \rangle.
\end{equation}
The set of monomials
\begin{equation}
\{ \text{monomials $m \in \FF[\xx_n]$} \,:\, m \notin \initial_<(I) \}
\end{equation}
descends to a $\FF$-basis for the quotient  ${\FF[\xx_n]} / {I}$; this basis is called the 
{\em standard monomial basis} (with respect to the term order $\leq$) \cite[Prop. 1, p. 230]{CLO}.

Let $I \subseteq \FF[\xx_n]$ be any ideal and let $\leq$ be a term order. 
A finite subset $G = \{g_1, \dots, g_r\} \subseteq I$ 
of nonzero polynomials in $I$ is called a {\em Gr\"obner basis} of $I$ if
\begin{equation}
\initial_<(I) = \langle \initial_<(g_1), \dots , \initial_<(g_r) \rangle.
\end{equation}
If $G$ is a Gr\"obner basis for $I$, then we have $I = \langle G \rangle$ \cite[Cor. 6, p. 77]{CLO}.

Let $G$ be a Gr\"obner basis for $I$ with respect to the
term order $\leq$.  The basis $G$ is called {\em minimal} if
\begin{itemize}
\item for any $g \in G$, the leading coefficient of $g$ with respect to $\leq$ is $1$, and
\item for any $g \neq g'$ in $G$, the leading monomial of $g$ does not divide the leading monomial of $g'$.
\end{itemize}
A minimal Gr\"obner basis $G$ is called {\em reduced} if in addition 
\begin{itemize}
\item  for any $g \neq g'$ in $G$, the leading monomial of $g$ does not divide any term in the polynomial $g'$.
\end{itemize}
Up to a choice of term order,
every ideal $I$ has a unique reduced Gr\"obner basis \cite[Prop. 6, p. 92]{CLO}.

\subsection{$\Sym$, $\QSym$, and $\NSym$}
Let $\xx = (x_1, x_2, \dots )$ be a totally ordered infinite set of variables 
and let $\Sym$ be the ($\ZZ$-)algebra of symmetric functions in $\xx$ with coefficients in $\ZZ$.
The algebra $\Sym$ is graded; its degree $n$ component has basis given by the collection
$\{s_{\lambda} \,:\, \lambda \vdash n\}$ of {\em Schur functions}.  The Schur function $s_{\lambda}$
may be defined as 
\begin{equation}
\label{schur-function-definition}
s_{\lambda} = \sum_T \xx^T,
\end{equation}
where the sum is over all semistandard tableaux $T$ of shape $\lambda$ and $\xx^T$ is the monomial
\begin{equation}
\xx^T := x_1^{\# \text{of 1s in $T$}} x_2^{\# \text{of 2s in $T$}} \cdots.
\end{equation}

Given partitions $\mu \subseteq \lambda$, we also let $s_{\lambda/\mu} \in \Sym$ denote the associated
{\em skew Schur function}.  The expansion of $s_{\lambda/\mu}$ in the $\xx$ variables is also given by 
Equation~\ref{schur-function-definition}.
In particular, if $\alpha$ is a composition (thought of as a ribbon), we have the
{\em ribbon Schur function} $s_{\alpha} \in \Sym$.

There is a coproduct of $\Sym$ given by replacing the variables $x_1,x_2,\ldots$ with $x_1,x_2,\ldots, y_1, y_2,\ldots$ such that $\Sym$ becomes a graded Hopf algebra which is self-dual under the basis $\{s_\lambda\}$~\cite[\S2]{GrinbergReiner}.

Let $\alpha = (\alpha_1, \dots, \alpha_k) \models n$ be a composition.  The {\em monomial quasisymmetric
function} is the formal power series
\begin{equation}
M_{\alpha} := \sum_{i_1 < \cdots < i_k} x_{i_1}^{\alpha_1} \cdots x_{i_k}^{\alpha_k}.
\end{equation}
The graded algebra of {\em quasisymmetric functions} $\QSym$ is the $\ZZ$-linear span of the $M_{\alpha}$,
where $\alpha$ ranges over all compositions.

We will focus on a basis for $\QSym$ other than the monomial quasisymmetric functions $M_{\alpha}$.
If $n$ is a positive integer and if $S \subseteq [n-1]$, the {\em Gessel fundamental quasisymmetric function}
$F_S$ attached to $S$ is
\begin{equation}
F_S := \sum_{\substack{i_1 \leq \cdots \leq i_n \\ j \in S \, \Rightarrow \, i_j < i_{j+1}}} x_{i_1} \cdots x_{i_n}.
\end{equation}
In particular, if $w \in \symm_n$ is a permutation with inverse descent set $\iDes(w) \subseteq [n-1]$,
we have the quasisymmetric function $F_{\iDes(w)}$.
If $\alpha \models n$ is a composition, we extend this notation by setting $F_{\alpha} := F_{\Des(\alpha)}$.

Next, let $\NSym$ be the graded algebra of {\em noncommutative symmetric functions}.
This is the free unital associative (noncommutative) algebra
\begin{equation}
\NSym := \ZZ \langle \mathbf{h}_1, \mathbf{h}_2, \dots \rangle
\end{equation}
generated over $\ZZ$  by the symbols
$\mathbf{h}_1, \mathbf{h}_2, \dots$, where $\mathbf{h}_d$ has degree $d$.
The degree $n$
component of  $\NSym$ has $\ZZ$-basis given by $\{ \mathbf{h}_{\alpha} \,:\, \alpha \models n\}$, where
for $\alpha = (\alpha_1, \dots, \alpha_{\ell}) \models n$ we set
\begin{equation}
\mathbf{h}_{\alpha} := \mathbf{h}_{\alpha_1}  \cdots \mathbf{h}_{\alpha_{\ell}}.
\end{equation}
Another basis of the degree $n$ piece of $\NSym$ 
 consists of the {\em ribbon Schur functions} $\{\sss_{\alpha} \,:\, \alpha \models n\}$.
 The ribbon Schur function $\sss_{\alpha}$ is defined by
 \begin{equation}
 \sss_{\alpha} := \sum_{\beta \preceq \alpha} (-1)^{\ell(\alpha) - \ell(\beta)} \mathbf{h}_{\beta}.
 \end{equation}

Finally, there are coproducts for $\QSym$ and $\NSym$ such that they become dual graded Hopf algebras~\cite[\S5]{GrinbergReiner}.

\subsection{Characteristic maps}
\label{char}
Let $A$ be a finite-dimensional algebra over a field $\FF$.
The {\em Grothendieck group $G_0(A)$ of the category of finitely-generated $A$-modules} is the quotient of the free abelian group generated by isomorphism classes $[M]$ of finitely-generated $A$-modules $M$ by the subgroup generated by  elements $[M] - [L] - [N]$ corresponding to short exact sequences $0 \rightarrow L \rightarrow M \rightarrow N \rightarrow 0$ of finitely-generated $A$-modules.
The abelian group $G_0(H_n(0))$ has free basis given by the collection of (isomorphism classes of) irreducible $A$-modules.
The {\em Grothendieck group $K_0(A)$ of the category of finitely-generated projective $A$-modules} is defined similarly, and has free basis given by the set of (isomorphism classes of) projective indecomposable $A$-modules.
If $A$ is semisimple then $G_0(A)=K_0(A)$.
See~\cite{ASS} for more details on representation theory of finite dimensional algebras.

The symmetric group algebra $\QQ[\symm_n]$ is semisimple and has irreducible representations $S^\lambda$ indexed by partitions $\lambda\vdash n$.  
The \emph{Grothendieck group} $G_0(\QQ[\symm_\bullet])$ of the tower $\QQ[\symm_\bullet]: \QQ[\symm_0]\hookrightarrow \QQ[\symm_1]\hookrightarrow \QQ[\symm_2] \hookrightarrow \cdots$ of symmetric group algebras is the direct sum of $G_0(\QQ[\symm_n])$ for all $n\ge0$.
It is a graded Hopf algebra with product and coproduct given by induction and restriction along the embeddings $\symm_m\otimes\symm_n\hookrightarrow\symm_{m+n}$.
The structual constants of $G_0(\QQ[\symm_\bullet])$ under the self-dual basis $\{S_\lambda\}$, where $\lambda$ runs through all partitions, are the well-known \emph{Littlewood-Richardson coefficients}. 


The {\em Frobenius character}
\footnote{The Frobenius character $\Frob(V)$ is indeed a ``character'' since the schur functions are characters of irreducible polynomial representations of the general linear groups.}
$\Frob(V)$ of a finite-dimensional $\QQ[\symm_n]$-module $V$ is
\begin{equation}
\Frob(V) := \sum_{\lambda \vdash n} \, [\, V : S^\lambda\, ] \cdot s_{\lambda} \in \Sym
\end{equation}
where $[\, V: S^\lambda \,]$ is the multiplicity of the simple module $S^\lambda$ among the composition factors of $V$.
The correspondence $\Frob$ gives a graded Hopf algebra isomorphism $G_0(\QQ[\symm_\bullet]) \cong \Sym$~\cite[\S4.4]{GrinbergReiner}.

One can refine $\Frob$ for graded representations of $\QQ[\symm_n]$.
Recall that the {\em Hilbert series} of a graded vector space $V = \bigoplus_{d \geq 0} V_d$ with each component $V_d$ finite-dimensional is
\begin{equation}
\Hilb(V;q) := \sum_{d \geq 0} \dim(V_d) \cdot q^d.
\end{equation}
If $V$ carries a graded action of $\QQ[\symm_n]$, we also define the {\em graded Frobenius series} by
\begin{equation}
\grFrob(V;q) := \sum_{d \geq 0} \Frob(V_d) \cdot q^d.
\end{equation}

Now let us recall the 0-Hecke analog of the above story.
Consider an arbitrary ground field $\FF$.
The representation theory of the $\FF$-algebra $H_n(0)$ was studied by Norton \cite{Norton}, 
which has a different flavor from that of $\QQ[\symm_n]$ since $H_n(0)$ is not semisimple.
Norton \cite{Norton} proved that 
the $H_n(0)$-modules 
\begin{equation}
P_{\alpha} := H_n(0) \pib_{w_0(\alpha)} \pi_{w_0(\alpha^c)},
\end{equation}
where $\alpha$ ranges over all compositions of $n$, form a complete list of nonisomorphic
 indecomposable projective 
$H_n(0)$-modules.  
For each $\alpha\models n$, the $H_n(0)$-module $P_\alpha$ has a basis 
\[ \{ \pib_w \pi_{w_0(\alpha^c)} : w\in\symm_n,\ \Des(w) = \Des(\alpha) \}. \]
Moreover, $P_\alpha$ has a unique maximal submodule spanned by all elements in the above basis except its cyclic generator $\pib_{w_0(\alpha)} \pi_{w_0(\alpha^c)}$, and the quotient of $P_\alpha$ by this maximal submodule, denoted by $C_\alpha$, is one-dimensional and admits an $H_n(0)$-action by $\pib_i = -1$ for all $i\in \Des(\alpha)$ and $\pib_i = 0$ for all $i\in \Des(\alpha^c)$. 
The collections $\{ P_\alpha:\alpha\models n\}$ and $\{ C_\alpha: \alpha\models n\}$ are complete lists of nonisomorphic projective indecomposable and irreducible $H_n(0)$-modules, respectively.

Just as the Frobenius character map gives a deep connection between the representation
theory of symmetric groups and the ring $\Sym$ of symmetric functions, there are two characteristic maps 
$\Ch$ and $\ch$, defined by Krob and Thibon \cite{KT}, which facilitate the study of representations of $H_n(0)$ through the rings 
$\QSym$ and $\NSym$.  Let us recall their construction.

The two Grothendieck groups $G_0(H_n(0))$ and $K_0(H_n(0))$ has free $\ZZ$-bases $\{P_\alpha:\alpha\models n\}$ and $\{C_\alpha:\alpha\models n\}$, respectively. 
Associated to the tower of algebras 
$H_{\bullet}(0) : H_0(0) \hookrightarrow H_1(0) \hookrightarrow H_2(0) \hookrightarrow \cdots$
are the two Grothendieck groups
\begin{center}
$G_0(H_{\bullet}(0)) := \bigoplus_{n \geq 0} G_0(H_n(0))$ and 
$K_0(H_{\bullet}(0)) := \bigoplus_{n \geq 0} K_0(H_n(0))$.
\end{center}
These groups are graded Hopf algebras with product and coproduct given by induction and restriction along the embeddings $H_n(0) \otimes H_m(0) \hookrightarrow H_{n+m}(0)$, and they are dual to each other via the pairing $\langle P_\alpha, C_\beta \rangle = \delta_{\alpha,\beta}$.

Analogously to the Frobenius correspondence, Krob and Thibon \cite{KT} defined two linear maps
\begin{center}
$\Ch: G_0(H_{\bullet}(0)) \rightarrow \QSym$ and $\ch: K_0(H_{\bullet}(0)) \rightarrow \NSym$
\end{center}
by $\Ch(C_{\alpha}) := F_{\alpha}$ and $\ch(P_{\alpha}) := \sss_{\alpha}$ for all compositions $\alpha$.  These maps
are isomorphisms of graded Hopf algebras.  
Krob and Thibon also showed \cite{KT} that for any composition $\alpha$, we have $\Ch(P_\alpha)$ equals the corresponding ribbon Schur function $s_{\alpha} \in \Sym$:
\begin{equation}
\label{ChP}
\Ch(P_{\alpha}) =  \sum_{w \in \symm_n \,:\, \Des(w) = \Des(\alpha)} F_{\iDes(w)} = s_{\alpha}.
\end{equation}

We give graded extensions of the maps $\Ch$ and $\ch$ as follows.  
Suppose that $V = \bigoplus_{d \geq 0} V_d$ is a graded $H_n(0)$-module with finite-dimensional homogeneous components $V_d$.
The {\em degree-graded noncommutative characteristic} and {\em degree-graded quasisymmetric characteristic} of $V$ are defined by
\begin{equation}
\ch_t(V) := \sum_{d \geq 0} \ch(V_d) \cdot t^d
\quad\text{and}\quad
\Ch_t(V) := \sum_{d \geq 0} \Ch(V_d) \cdot t^d.
\end{equation}

On the other hand, the 0-Hecke algebra $H_n(0)$ has a {\em length filtration}
$H_n(0)^{(0)} \subseteq H_n(0)^{(1)} \subseteq H_n(0)^{(2)} \subseteq \cdots$ 
where $H_n(0)^{(\ell)}$ is the span of $\{ \pi_w: w \in \symm_n,\ \ell(w) \ge \ell \}$.
If $V = H_n(0)v$ is a cyclic $H_n(0)$-module with a distinguished
generator $v$, we get an induced length filtration of $V$ given by
\begin{equation}
V^{(\ell)} := H_n(0)^{(\ell)} v.
\end{equation}
Following Krob and Thibon~\cite{KT}, we define the \emph{length-graded quasymmetric characteristic} of $V$ as
\begin{equation}
\Ch_q(V) := \sum_{\ell \geq 0}  \Ch \left( {V^{(\ell)}} \big/ {V^{(\ell+1)}} \right) \cdot q^{\ell}.
\end{equation}


Now suppose $V = \bigoplus_{d \geq 0} V_d$ is a graded $H_n(0)$-module which is also cyclic. 
We get a bifiltration of $V$ consisting of the $H_n(0)$-modules 
$V^{(\ell,d)} := V^{(\ell)} \cap V_d$ for $\ell, d \geq 0$.  
The {\em length-degree-bigraded quasisymmetric characteristic} of $V$ is then
\begin{equation}
\Ch_{q,t}(V) := \sum_{\ell, d \geq 0} \Ch \left( {V^{(\ell,d)}} \Big/ \Big( {V^{(\ell+1,d)} + V^{(\ell,d+1)}} \Big) \right) \cdot q^{\ell} t^d.
\end{equation}
More generally, if a graded $H_n(0)$-module $V$ is isomorphic to a direct sum of cyclic $H_n(0)$-modules, then define $\Ch_{q,t}(V)$ by applying $\Ch_{q,t}$ to its direct summands.
Note that this may depend on the choice of the direct sum decomposition of $V$ into cyclic modules.
Specializing $q=1$ and $t=1$ gives $\Ch_{1,t}(V)=\Ch_t(V)$ and $\Ch_{q,1}(V)=\Ch_q(V)$, respectively.

\section{Hilbert Series and Artin basis}
\label{Hilbert}

\subsection{The point sets $Z_{n,k}$}
In this section we will prove that $\dim(S_{n,k}) = |\OP_{n,k}|$.
To do this, we will use tools from elementary algebraic geometry.
This basic method dates back to the work of Garsia and Procesi on Springer
fibers and Tanisaki quotients \cite{GP}.

Given a finite point set $Z \subseteq \FF^n$, let $\II(Z) \subseteq \FF[\xx_n]$ be the ideal of 
polynomials which vanish on $Z$:
\begin{equation}
\II(Z) := \{ f \in \FF[\xx_n] \,:\, f(\zz) = 0 \text{ for all $\zz \in Z$} \}.
\end{equation} 
There is a natural identification of the quotient ${\FF[\xx_n]} / {\II(Z)}$ with the collection of
polynomial functions $Z \rightarrow \FF$.

We claim that any function $Z \rightarrow \FF$ may be realized as the restriction of a polynomial
function.  This essentially follows from Lagrange Interpolation.  Indeed, since $Z \subseteq \FF^n$ is finite, there exist
field elements $\alpha_1, \dots, \alpha_m \in \FF$ such that $Z \subseteq \{\alpha_1, \dots, \alpha_m\}^n$. 
For any $n$-tuple of integers $(i_1, \dots, i_n)$ between $1$ and $m$, the polynomial
\begin{equation*}
\prod_{j_1 \neq i_1} (x_1 - \alpha_{j_1}) \cdots \prod_{j_n \neq i_n} (x_n - \alpha_{j_n}) \in \FF[\xx_n]
\end{equation*} 
vanishes on every point of $\{\alpha_1, \dots, \alpha_m\}^n$ except for 
$(\alpha_{i_1}, \dots, \alpha_{i_n})$.  Hence, an arbitrary $\FF$-valued function on
$\{\alpha_1, \dots, \alpha_m\}^n$ may be realized using a linear combination of polynomials of the above form.
Since $Z \subseteq \{\alpha_1, \dots, \alpha_m\}^n$, the same is true for an arbitrary $\FF$-valued function on $Z$.

By the last paragraph, we may identify the quotient ${\FF[\xx_n]} / {\II(Z)}$ with the collection of all 
functions $\ZZ \rightarrow \FF$.
In particular, the dimension of this quotient as an $\FF$-vector space is
\begin{equation}
\dim \left( {\FF[\xx_n]} / {\II(Z)} \right) = |Z|.
\end{equation}

The ideal $\II(Z)$ is almost never homogeneous.  To get a homogeneous ideal, we do the following.
For any nonzero polynomial $f \in \FF[\xx_n]$, let $\tau(f)$ be the top degree component of $f$.  That 
is, if $f = f_d + f_{d-1} + \cdots + f_0$ where $f_i$ has homogeneous degree $i$ for all $i$ and $f_d \neq 0$, then
$\tau(f) = f_d$.  The ideal $\TT(Z) \subseteq \FF[\xx_n]$ is generated by the top degree components of all
nonzero polynomials in $\II(Z)$.  In symbols:
\begin{equation}
\TT(Z) := \langle \tau(f) \,:\, f \in \II(Z) - \{0\} \rangle.
\end{equation}

The ideal $\TT(Z)$ is homogeneous by definition, so that ${\FF[\xx_n]} / {\TT(Z)}$ is a graded $\FF$-vector space.
Moreover, it is well known that
\begin{equation}
\dim \left( {\FF[\xx_n]} / {\TT(Z)} \right) = \dim \left( {\FF[\xx_n]} / {\II(Z)} \right) = |Z|.
\end{equation}

Our three-step strategy for proving $\dim(S_{n,k}) = |\OP_{n,k}|$ is as follows.

\begin{enumerate}
\item  Find a  point set $Z_{n,k} \subseteq \FF^n$ which is in bijective correspondence with $\OP_{n,k}$.
\item  Prove that the generators of $J_{n,k}$ arise as top degree components of polynomials in $\II(Z_{n,k})$,
so that $J_{n,k} \subseteq \TT(Z_{n,k})$.
\item Use Gr\"obner theory to prove $\dim(S_{n,k}) \leq |\OP_{n,k}|$, forcing $\dim(S_{n,k}) = |\OP_{n,k}|$
by Steps 1 and 2.
\end{enumerate}

A similar three-step strategy was used by Haglund, Rhoades, and Shimozono \cite{HRS} 
in their analysis of the ungraded $\symm_n$-module structure of $R_{n,k}$.
In our setting, since we do not have a group action, we can only use this strategy to deduce the graded vector
space structure of $S_{n,k}$, rather than the $H_n(0)$-module structure of $S_{n,k}$.

To achieve Step 1 of our strategy, we need to find a candidate set $Z_{n,k} \subseteq \FF^n$ which is in bijective
correspondence with $\OP_{n,k}$.  Here we run into a problem:  to define our candidate point sets, we need the 
field $\FF$ to contain at least $n+k-1$ elements.  This problem did not arise in the work of Haglund et. al. \cite{HRS};
they worked exclusively over the field $\QQ$.  To get around this problem, we use the following  trick.

\begin{lemma}
\label{algebraic-swindle}
Let $\FF \subseteq \KK$ be a field extension and $J = \langle f_1, \dots, f_r \rangle \subseteq \FF[\xx_n]$ an ideal of $\FF[\xx_n]$ generated by $f_1, \dots, f_r \in \FF[\xx_n]$.  
Then $\dim_\FF ( {\FF[\xx_n]} / {J} )  = \dim_\KK ( {\KK[\xx_n]} / J' )$ where $J' := {\KK\otimes_\FF J}$.
\end{lemma}

Since $J_{n,k}$ is generated by polynomials with all coefficients equal to $1$, the generating set of $J_{n,k}$ 
satisfies the conditions of Lemma~\ref{algebraic-swindle}.

\begin{proof}
Let $\leq$ be any term order.
It suffices to show that the quotient rings ${\FF[\xx_n]} / {J}$ and ${\KK[\xx_n]} / {J'}$ have the same 
standard monomial bases with respect to $\leq$.  
To calculate the reduced Gr\"obner basis for the ideal $J$, we apply Buchberger's Algorithm~\cite[Ch. 2, \S7]{CLO} 
to the generating set $\{f_1, \dots, f_r\}$.   
To calculate the Gr\"obner basis for the ideal $J'$, we also apply Buchberger's Algorithm to the generating set
$\{f_1, \dots, f_r\}$.  In either case, all of the coefficients involved in the polynomial long division are contained in the
field $\FF$.  In particular, the reduced Gr\"obner bases of $J$ and $J'$ coincide.  Hence, the standard monomial
bases of ${\FF[\xx_n]} / {J}$ and ${\KK[\xx_n]} / {J'}$ also coincide.
\end{proof}

We are ready to define our point sets $Z_{n,k}$.  Thanks to Lemma~\ref{algebraic-swindle}, we may harmlessly
assume that 
the field $\FF$ has at least $n+k-1$ elements by replacing $\FF$ with an extension if necessary.
We will have to choose a somewhat non-obvious point set $Z_{n,k} \subseteq \FF^n$ in order to get the 
desired equality of ideals $\TT(Z_{n,k}) = J_{n,k}$.

\begin{defn}
Assume $\FF$ has at least $n+k-1$ elements and let $\alpha_1, \alpha_2, \dots, \alpha_{n+k-1} \in \FF$
be a list of $n+k-1$ distinct field elements.  Define $Z_{n,k} \subseteq \FF^n$ to be the collection of points
$(z_1, z_2, \dots, z_n)$ such that
\begin{itemize}
\item  for $1 \leq i \leq n$ we have $z_i \in \{\alpha_1, \alpha_2, \dots, \alpha_{k+i-1} \}$,
\item the coordinates $z_1, z_2, \dots, z_n$ are distinct, and
\item  we have $\{\alpha_1, \alpha_2, \dots, \alpha_k\} \subseteq \{z_1, z_2, \dots, z_n\}$.
\end{itemize}
\end{defn}

We claim that $Z_{n,k}$ is in bijective correspondence with $\OP_{n,k}$.  A bijection
$\varphi: \OP_{n,k} \rightarrow Z_{n,k}$ may be obtained as follows.
Let $\sigma = (B_1 \mid \cdots \mid B_k) \in \OP_{n,k}$ be an ordered set partition;
we define $\varphi(\sigma) = (z_1, \dots, z_n) \in Z_{n,k}$.
For $1 \leq i \leq k$, we first set $z_j = \alpha_i$, where $j = \min(B_i)$.  
Write the set of unassigned indices of
$(z_1, \dots, z_n)$ as  $S = [n] - \{\min(B_1), \dots, \min(B_k)\} = \{s_1 < \cdots < s_{n-k} \}$.
For $1 \leq j \leq n-k$ let $\ell_j$ be the number of blocks $B$ weakly to the left of
$s_j$ in $\sigma$ which satisfy $\min(B) < s_j$. 
Let $z_{s_1} = \alpha_{k + \ell_1}$.  Assuming $z_{s_1}, z_{s_2}, \dots, z_{s_{j-1}}$ have already 
been defined, let $z_{s_j}$ be the $\ell_j^{th}$ term in the sequence formed by deleting
$z_{s_1}, z_{s_2}, \dots, z_{s_{j-1}}$ from the sequence
$(\alpha_{k+1}, \alpha_{k+2}, \dots, \alpha_{k-n+1})$.

As an example of the map $\varphi$, let $\sigma = ( 7 8  \mid 2 3 6  \mid 1 4  \mid 5 9) \in \OP_{9,4}$.
The point $\varphi(\sigma) = (z_1, \dots, z_9)$ is obtained by first setting $z_7 = \alpha_1,
z_2 = \alpha_2, z_1 = \alpha_3,$ and $z_5 = \alpha_4$.  The set $S$ is
$S = \{s_1 < s_2 < s_3 < s_4 < s_5\} = \{3 < 4 < 6 < 8  < 9\}$.  We have $\ell_1 = 1, \ell_2 = 2, \ell_3 = 1, \ell_4 = 1,$ and
$\ell_5 = 4$, so that $z_{s_1} = z_3 = \alpha_5, z_{s_2} = z_4 = \alpha_7, z_{s_3} = z_6 = \alpha_6,
z_{s_4} = z_8 = \alpha_8,$ and $z_{s_5} = z_9 = \alpha_{12}$.  In summary, we have
\begin{equation*}
\varphi: ( 7 8  \mid 2 3 6  \mid 1 4  \mid 5 9) \mapsto 
(\alpha_3  ,  \alpha_2 ,  \alpha_5 , \alpha_7  ,  \alpha_4 , \alpha_6  ,  \alpha_1 , \alpha_8  , \alpha_{12}  ).
\end{equation*}

We leave it for the reader to check that $\varphi: \OP_{n,k} \rightarrow Z_{n,k}$ is 
well-defined and invertible.  The point set $Z_{n,k}$ therefore
achieves Step 1 of our strategy.

Achieving Step 2 of our strategy involves showing that the generators of $J_{n,k}$ arise as 
top degree components of strategically chosen polynomials vanishing on $Z_{n,k}$.

\begin{lemma}
\label{j-contained-in-t}
Assume $\FF$ has at least $n+k-1$ elements.  We have $J_{n,k} \subseteq \TT(Z_{n,k})$.
\end{lemma}

\begin{proof}
It suffices to show that every generator of $J_{n,k}$ arises as the top degree component of a polynomial
in $\II(Z_{n,k})$.  Let us first consider the generators $h_k(x_1), h_k(x_1, x_2), \dots, h_k(x_1, x_2, \dots, x_n)$.

For $1 \leq i \leq n$, we claim that
\begin{equation}
\label{i-membership}
\sum_{j \geq 0} (-1)^j h_{k-j}(x_1, x_2, \dots, x_i) e_j(\alpha_1, \alpha_2, \dots, \alpha_{k+i-1}) \in \II(Z_{n,k}).
\end{equation}
Indeed, this alternating sum is the coefficient of $t^k$ in the power series expansion of the rational function
\begin{equation}
\frac{(1-\alpha_1 t)(1 - \alpha_2 t) \cdots (1 - \alpha_{k+i-1} t)}{(1 - x_1 t) (1 - x_2 t) \cdots (1 - x_i t)}.
\end{equation}
If $(x_1, \dots, x_n) \in Z_{n,k}$, by the definition of $Z_{n,k}$ the
 terms in the denominator cancel with $i$ terms in the numerator, yielding a polynomial
in $t$ of degree $k-1$.  
The assertion (\ref{i-membership}) follows.
Taking the highest degree component, we get
$h_k(x_1, x_2, \dots, x_i) \in \TT(Z_{n,k})$.

Next, we show $e_r(x_1, \dots, x_n) \in \TT(Z_{n,k})$ for $n-k < r \leq n$.  To prove this, we claim that
\begin{equation}
\label{i-membership-two}
\sum_{j \geq 0} (-1)^j e_{r-j}(x_1, \dots, x_n) h_j(\alpha_1, \dots, \alpha_{n+k-1}) \in \II(Z_{n,k}).
\end{equation}
Indeed, this alternating sum is the coefficient of $t^r$ in the rational function
\begin{equation}
\frac{(1 + x_1 t) (1 + x_2 t) \cdots (1 + x_n t)}{(1 + \alpha_1 t) (1 + \alpha_2 t) \cdots (1 + \alpha_k t)}.
\end{equation}
If $(x_1, \dots, x_n) \in Z_{n,k}$, the terms in the denominator cancel with $k$ terms in the numerator, yielding a polynomial
in $t$ of degree $n-k$.  Since $r > n-k$, the assertion (\ref{i-membership-two}) follows.
Taking the highest degree component, we get
$e_r(x_1, \dots, x_n) \in \TT(Z_{n,k})$.
\end{proof}

\subsection{The Hilbert series of $S_{n,k}$}
Let $<$ be the $\neglex$ term order on $\FF[\xx_n]$.
We are ready to execute Step 3 of our strategy and describe the standard monomial basis 
of the quotient $S_{n,k}$.  To do so, we recall the definition of `skip monomials' in $\FF[\xx_n]$ of \cite{HRS}.

Let $S = \{s_1 < \cdots < s_m \} \subseteq [n]$ be a set.  Following \cite[Defn. 3.2]{HRS},
the {\em skip monomial} $\xx(S)$ is the monomial in $\FF[\xx_n]$ given by
\begin{equation}
\xx(S) := x_{s_1}^{s_1} x_{s_2}^{s_2 - 1} \cdots x_{s_m}^{s_m-m+1}.
\end{equation}
For example, we have $\xx(2578) = x_2^2 x_5^4 x_7^5 x_8^5$.  The adjective `skip' refers to the fact that 
the exponent sequence $\xx(S)$ increases whenever the set $S$ skips a letter.
Our variable order convention will require us to consider the {\em reverse skip monomial}
\begin{equation}
\xx(S)^* := x_{n-s_1+1}^{s_1} x_{n-s_2+1}^{s_2-1} \cdots x_{n-s_m+1}^{s_m-m+1}.
\end{equation}
For example, if $n = 9$ we have $\xx(2578)^* = x_8^2 x_5^4 x_3^5 x_2^5$. 
The following definition is the reverse of \cite[Defn. 4.4]{HRS}.

\begin{defn}
\label{nonskip}
Let $k \leq n$ be positive integers.  A monomial $m \in \FF[\xx_n]$ is {\em $(n,k)$-reverse nonskip} if
\begin{itemize}
\item $x_i^k \nmid m$ for $1 \leq i \leq n$, and
\item $\xx(S)^* \nmid m$ for all $S \subseteq [n]$ with $|S| = n+k-1$.
\end{itemize}
Let $\CCC_{n,k}$ denote the collection of all $(n,k)$-reverse nonskip monomials in $\FF[\xx_n]$.
\end{defn}

There is some redundancy in Definition~\ref{nonskip}.  In particular, if $n \in S$, the power of $x_1$
in $\xx(S)^*$  where $|S| = n-k+1$ is $x_1^k$, so that we need only consider those sets $S$ with $n \notin S$.

\begin{theorem}
\label{standard-monomial-basis}
Let $\FF$ be any field and $\leq$ be the $\neglex$ term order on $\FF[\xx_n]$.

The standard monomial basis of $S_{n,k} = {\FF[\xx_n]} / {J_{n,k}}$ with respect to $\leq$ is $\CCC_{n,k}$.
\end{theorem}

\begin{proof}
By the definition of $\neglex$, we have
\begin{equation}
\initial_<(h_k(x_1, x_2, \dots, x_i)) = x_i^k \in \initial_<(J_{n,k}).
\end{equation}
By \cite[Lem. 3.4, Lem. 3.5]{HRS} we also have $\xx(S)^* \in \initial_<(J_{n,k})$ whenever $S \subseteq [n]$
satisfies $|S| = n-k+1$.  It follows that $\CCC_{n,k}$ contains the standard monomial basis of $S_{n,k}$.

To prove that $\CCC_{n,k}$ is the standard monomial basis of $S_{n,k}$, it suffices to show
$|\CCC_{n,k}| \leq \dim(S_{n,k})$.  Thanks to Lemma~\ref{algebraic-swindle}, we may replace $\FF$ by an extension
if necessary to assume that $\FF$ contains at least $n+k-1$ elements.
By Lemma~\ref{j-contained-in-t}, we have 
\begin{equation}
\dim(S_{n,k}) = \dim \left( {\FF[\xx_n]} / {J_{n,k}} \right) 
\geq \dim \left( {\FF[\xx_n]} / {\TT(Z_{n,k})} \right) = |Z_{n,k}| = |\OP_{n,k}|.
\end{equation}
On the other hand, \cite[Thm. 4.9]{HRS} implies (after reversing variables) that
$|\OP_{n,k}| = |\CCC_{n,k}|$.
\end{proof}

When $k = n$, the collection $\CCC_{n,n}$ consists of sub-staircase monomials 
$x_1^{a_1} \cdots x_n^{a_n}$ whose exponent sequences satisfy $0 \leq a_i \leq n-i$; 
this is the basis for the coinvariant algebra obtained by E. Artin~\cite{Artin} using Galois theory.
Let us mention an analogous characterization of $\CCC_{n,k}$ which was derived in \cite{HRS}.

Recall that a {\em shuffle} of two sequences $(a_1, \dots, a_r)$ and $(b_1, \dots, b_s)$ 
is an interleaving $(c_1, \dots, c_{r+s})$ of these sequences which preserves the relative order
of the $a$'s and the $b$'s.  
A {\em $(n,k)$-staircase} is a shuffle of the sequences
$(k-1, k-2, \dots, 1, 0)$ and $(k-1, k-1, \dots, k-1)$, where the second sequence has $n-k$ copies of $k-1$.
For example, the $(5,3)$-staircases are the shuffles of $(2,1,0)$ and $(2,2)$:
\begin{center}
$(2,1,0,2,2), (2,1,2,0,2), (2,2,1,0,2), (2,1,2,2,0), (2,2,1,2,0),$ and $(2,2,2,1,0)$.
\end{center}
The following theorem is just the reversal of \cite[Thm. 4.13]{HRS}.

\begin{corollary}
\label{staircase-standard}  (\cite[Thm. 4.13]{HRS})
The monomial basis $\CCC_{n,k}$ of $S_{n,k}$ is the set of monomials $x_1^{a_1} x_2^{a_2} \cdots x_n^{a_n}$
in $\FF[\xx_n]$  whose exponent sequences $(a_1, a_2, \dots, a_n)$
are componentwise $\leq$ {\em some} $(n,k)$-staircase.
\end{corollary}

For example, consider the case $(n,k) = (4,2)$.  The $(4,2)$-staircases are the shuffles of $(1,0)$ and $(1,1)$:
\begin{center}
$(1,0,1,1), (1,1,0,1),$ and $(1,1,1,0)$.
\end{center}
It follows that 
\begin{equation*}
\CCC_{4,2} = \{1, x_1, x_2, x_3, x_4, x_1 x_2, x_1 x_3, x_1 x_4, x_2 x_3, x_2 x_4, x_3 x_4, x_1 x_3 x_4, 
x_1 x_2 x_4, x_1 x_2 x_3 \}
\end{equation*}
is the standard monomial basis of $S_{4,2}$ with respect to $\neglex$.  Consequently, we have the Hilbert series
\begin{equation*}
\Hilb(S_{4,2};q) = 1 + 4q + 6q^2 + 3q^3.
\end{equation*}

We can also describe a Gr\"obner basis of the ideal $J_{n,k}$.
For $\gamma = (\gamma_1, \dots, \gamma_n)$  a {\em weak composition} (i.e., possibly containing $0$'s)
of length $n$, let $\kappa_{\gamma}(\xx_n) \in \FF[\xx_n]$ be the associated {\em Demazure character}
(see e.g. \cite[Sec. 2.4]{HRS}).  

If $S \subseteq [n]$, let $\gamma(S) = (\gamma_1, \dots, \gamma_n)$ be the exponent sequence of the corresponding
skip monomial $\xx(S)$.  That is, if $S = \{s_1 < \cdots < s_m\}$ we have
\begin{equation}
\gamma_i = \begin{cases}
s_j - j + 1 & \text{if $i = s_j \in S$} \\
0 & \text{if $i \notin S$}.
\end{cases}
\end{equation}
Let $\gamma(S)^* = (\gamma_n, \dots, \gamma_1)$ be the reverse of the weak composition $\gamma(S)$.
In particular, we can consider the Demazure character $\kappa_{\gamma(S)^*}(\xx_n) \in \FF[\xx_n]$.

\begin{theorem}
\label{groebner-basis}
Let $k \leq n$ be positive integers and
let $\leq$ be the $\neglex$ term order on $\FF[\xx_n]$.

The polynomials
\begin{equation*}
h_k(x_1), h_k(x_1, x_2), \dots, h_k(x_1, x_2, \dots, x_n)
\end{equation*}
together with the Demazure characters
\begin{equation*}
\kappa_{\gamma(S)^*}(\xx_n) \in \FF[\xx_n]
\end{equation*}
for all $S \subseteq [n-1]$ satisfying $|S| = n-k+1$, form a Gr\"obner basis for the ideal $J_{n,k}$.

When $k < n$, this Gr\"obner basis is minimal.
\end{theorem}

For example, if $(n,k) = (6,4)$, a Gr\"obner basis of $J_{6,4} \subseteq \FF[\xx_6]$ is given by the polynomials
\begin{center}
$h_4(x_1), h_4(x_1, x_2), h_4(x_1, x_2, x_3), h_4(x_1, x_2, x_3, x_4), h_4(x_1, x_2, x_3, x_4, x_5),
h_4(x_1,x_2,x_3,x_4,x_5,x_6)$
\end{center}
together with the Demazure characters
\begin{center}
$\kappa_{(0,0,0,1,1,1)}(\xx_6), \kappa_{(0,0,2,0,1,1)}(\xx_6), \kappa_{(0,3,0,0,1,1)}(\xx_6),
\kappa_{(0,0,2,2,0,1)}(\xx_6), \kappa_{(0,3,0,2,0,1)}(\xx_6),$ \\
$\kappa_{(0,3,3,0,0,1)}(\xx_6), \kappa_{(0,0,2,2,2,0)}(\xx_6), \kappa_{(0,3,0,2,2,0)}(\xx_6),
\kappa_{(0,3,3,0,2,0)}(\xx_6), \kappa_{(0,3,3,3,0,0)}(\xx_6)$.
\end{center}

\begin{proof}
We need to show that the polynomials in question lie in the ideal $J_{n,k}$.  This is clear for the polynomials
$h_k(x_1, \dots, x_i)$.  For the Demazure characters, we apply \cite[Lem. 3.4]{HRS} 
(and in particular \cite[Eqn. 3.4]{HRS}) to see that $\kappa_{\gamma(S)^*}(\xx_n) \in J_{n,k}$ whenever 
$S \subseteq [n-1]$ satisfies $|S| = n-k+1$.

Next we examine the leading terms of the polynomials in question.  It is evident that
\begin{equation*}
\initial_<(h_k(x_1, \dots, x_i)) = x_i^k.  
\end{equation*}
After applying variable reversal to \cite[Lem. 3.5]{HRS}, we see that
\begin{equation*}
\initial_<(\kappa_{\gamma(S)^*}(\xx_n)) = \xx(S)^*.
\end{equation*}
By Theorem~\ref{standard-monomial-basis} and the remarks following Definition~\ref{nonskip}, it follows
that these monomials generate the initial ideal $\initial_<(J_{n,k})$ of $J_{n,k}$.

When $k < n$, observe that for $S \subseteq [n-1]$ with $|S| = n-k+1$, the monomial 
$\xx(S)^*$ has support $\{ i \,:\, n-i+1 \in S \}$.  Moreover, the monomial 
$\xx(S)^*$ does not contain any exponents $\geq k$ since $n \notin S$.  
The minimality of the Gr\"obner basis follows.
\end{proof}

Theorem~\ref{groebner-basis} is the 0-Hecke analog of \cite[Thm. 4.14]{HRS}.  
Unlike the case of \cite[Thm. 4.14]{HRS}, the Gr\"obner basis of Theorem~\ref{groebner-basis}
is not reduced.  The authors do not have a conjecture for the reduced Gr\"obner basis
for the ideal $J_{n,k}$.
The work of \cite{HRS} gives us a formula for the Hilbert series of $S_{n,k}$.

\begin{theorem}
\label{hilbert-series}
Let $k \leq n$ be positive integers.  We have
$\Hilb(S_{n,k};q) = \rev_q([k]!_q \cdot \Stir_q(n,k))$.
\end{theorem}

\begin{proof}
By Theorem~\ref{standard-monomial-basis} and \cite[Thm. 4.13]{HRS}, the Hilbert series of $S_{n,k}$ equals the Hilbert series of $R_{n,k}$. 
Applying \cite[Thm. 4.10]{HRS} finishes the proof.
\end{proof}

\section{Garsia-Stanton type bases}
\label{Garsia}

Let $k \leq n$ be positive integers.
Given a composition $\alpha \models n$ and a length $n$ sequence $\ii = (i_1, \dots, i_n)$ 
of nonnegative integers, define a monomial $\xx_{\alpha,\ii} \in \FF[\xx_n]$ by
\begin{equation}
\xx_{\alpha,\ii} := \left( \prod_{j \in \Des(\alpha)} x_1 x_2 \cdots x_j \right) x_1^{i_1} x_2^{i_2} \cdots x_{n}^{i_{n}}.
\end{equation}
If $w \in \symm_n$ is a permutation and $\ii = (i_1, \dots, i_n)$ is a sequence of nonnegative integers, we 
define the {\em generalized Garsia-Stanton monomial}
$gs_{w,\ii} := w(\xx_{\alpha,\ii})$, where $\alpha \models n$ is characterized by $\Des(\alpha) = \Des(w)$.
The degree of $gs_{w,\ii}$ is given by $\deg(gs_{w,\ii}) = \maj(w) + |\ii|$, where
$|\ii| := i_1 + \cdots  + i_n$.

For example, let $(n,k) = (9,5)$, $w = 254689137 \in \symm_9$ and $\ii = (2,2,1,1,0,0,0,0,0)$.  We have $\Des(w) = \{2,6\}$, so that the composition $\alpha \models 9$ with
$\Des(\alpha) = \Des(w)$ is $\alpha = (2,4,3)$.  It follows that
\begin{equation*}
\xx_{\alpha,\ii} = (x_1 x_2) (x_1 x_2 x_3 x_4 x_5 x_6) (x_1^2 x_2^2 x_3^1 x_4^1).
\end{equation*}
The corresponding generalized GS monomial is
\begin{equation*}
gs_{w,\ii} = (x_2 x_5) (x_2 x_5 x_4 x_6 x_8 x_9) (x_2^2 x_5^2 x_4^1 x_6^1).
\end{equation*}

Haglund, Rhoades, and Shimozono introduced \cite[Defn. 5.2]{HRS} (using different notation) the following 
collection $\GS_{n,k}$ of monomials:
\begin{equation}
\GS_{n,k} := \{ gs_{w,\ii} \,:\, w \in \symm_n, \, k - \des(w) > i_1 \geq \cdots \geq i_{n-k} \geq 0 = i_{n-k+1} = \cdots = i_n\}.
\end{equation}
When $k = n$, we have $gs_{w,\ii}\in\GS_{n,n}$ if and only if $w\in\symm_n$ and $\ii=0^n$ is the sequence of $n$ zeros.
Garsia \cite{Garsia} proved that $\GS_{n,n}$ descends to a basis of the classical coinvariant algebra $R_n$.
Garsia and Stanton \cite{GS} later studied $\GS_{n,n}$ in the context of Stanley-Reisner theory. 
Extending Garsia's result,
Haglund et. al. proved that $\GS_{n,k}$ descends to a basis of $R_{n,k}$ \cite[Thm. 5.3]{HRS}.
We will prove that $\GS_{n,k}$ also descends to a basis of $S_{n,k}$.  In fact, we will 
prove that $\GS_{n,k}$ is just one of a family of bases of $S_{n,k}$.

Huang used isobaric Demazure operators to define a basis of the classical coinvariant algebra $R_n$
which is related to the classical GS basis $\GS_{n,n}$ by a unitriangular transition matrix \cite{Huang}.
We will modify $\GS_{n,k}$ to get a new basis of $S_{n,k}$ in an analogous way.
As in \cite{Huang}, our modified basis will be crucial in our analysis of the $H_n(0)$-structure of $S_{n,k}$.
This modified basis and $\GS_{n,k}$ itself will both belong to the following family of bases of $S_{n,k}$.

\begin{lemma}
\label{garsia-stanton-type-bases}
Let $\leq$ be the $\neglex$ term order on $\FF[\xx_n]$.
Let $\BBB_{n,k} = \{b_{w,\ii}\}$ be a set of polynomials indexed by pairs $(w, \ii)$ where $w \in \symm_n$
and $\ii = (i_1, \dots, i_n)$ satisfy 
\begin{equation*}
k - \des(w) > i_1 \geq \cdots \geq i_{n-k} \geq 0 = i_{n-k+1} = \cdots = i_n.  
\end{equation*}
Assume that any $b_{w,\ii} \in \BBB_{n,k}$
has the form
\begin{equation}
b_{w,\ii} = gs_{w,\ii} + \sum_{m < gs_{w,\ii}} c_m \cdot m,
\end{equation}
where the $c_m \in \FF$ are scalars which could depend on $(w,\ii)$.

We have that $\BBB_{n,k}$ descends to a basis of $S_{n,k}$.
\end{lemma}


\begin{proof}
By \cite[Thm. 5.3]{HRS}, we know that $|\BBB_{n,k}| = |\GS_{n,k}| = |\OP_{n,k}|$.  By Theorem~\ref{hilbert-series}, we have
$\dim(S_{n,k}) = |\OP_{n,k}|$.  Therefore, it is enough to show that $\BBB_{n,k}$ descends to a spanning set of
$S_{n,k}$.

If $\BBB_{n,k}$ did not descend to a spanning set of $S_{n,k}$, then there would be a monomial 
$m \in \FF[\xx_n]$ whose image $m + J_{n,k}$ did not lie in the span of $\BBB_{n,k}$.
Working towards a contradiction, suppose that such a monomial 
existed.

Let $m = x_1^{a_1} \cdots x_n^{a_n}$ be any monomial in $\FF[\xx_n]$.  We argue that $m$ is expressible 
modulo $J_{n,k}$ as a linear combination of monomials of the form $m' = x_1^{b_1} \cdots x_n^{b_n}$
with $b_i < k$ for all $i$.  Indeed, if $m$ does not already have this form, choose $i$ maximal such that
$a_i > k$.  Since $h_k(x_1, \dots, x_i) \in J_{n,k}$, modulo $J_{n,k}$ we have
\begin{equation}
m \equiv -(x_1^{a_1} \cdots x_i^{a_i - k} \cdots x_n^{a_n}) \sum_{\substack{1 \leq j_1 \leq \cdots \leq j_k \leq i \\ j_1 \neq i}}
x_{j_1} \cdots x_{j_k}.
\end{equation}
If every monomial appearing on the right hand side is of the required form, we are done.  
Otherwise, we may iterate this procedure.  Since $h_k(x_1) = x_1^k \in J_{n,k}$, iterating this procedure 
eventually yields $0$ or a linear combination of monomials of the required form. 

Let $m = x_1^{a_1} \cdots x_n^{a_n}$ be any monomial in $\FF[\xx_n]$ such that $m + J_{n,k}$ does not 
lie in the span of $\BBB_{n,k}$.
By the last paragraph, we may assume that $a_i < k$ for all 
$1 \leq i \leq n$.  
Choose such an $m$ which is minimal in the lexicographic term ordering 
(with respect to the usual variable order $x_1 > x_2 > \cdots > x_n$).
We associate a collection of objects to $m$ as follows.  

Let $\lambda(m) = \mathrm{sort}(a_1, \dots, a_n)$ be the 
sequence obtained by sorting the exponent sequence $a_1, \dots, a_n$ into weakly decreasing order.
Let $\sigma(m)=\sigma_1\cdots \sigma_n\in\symm_n$
 be the permutation (in one-line notation) obtained by listing the indices of variables in 
weakly decreasing order of the exponents in $m$, breaking ties by listing smaller indexed variables first.
Let $d(m) = (d_1, \dots, d_n)$ be the integer sequence given by 
$d_j = |\Des(\sigma(m)) \cap \{j, j+1, \dots, n\} |$.  Finally, it was shown 
by Adin, Brenti, and Roichman \cite{ABR} that the componentwise difference
$\lambda(m) - d(m)$ is an integer partition (i.e., has weakly decreasing components).
Let $\mu(m)$ be the {\em conjugate} of this integer partition.

Adin, Brenti, and Roichman \cite[Lem. 3.5]{ABR} proved that we can `straighten' the monomial $m$
and write
\begin{equation}
m = gs_{\sigma(m)} e_{\mu(m)}(\xx_n) - \Sigma,
\end{equation}
where $\Sigma$ is a linear combination of monomials which are $< m$ in lexicographical order
and do not involve any exponents which are $\geq k$.  
Here 
\begin{equation}
gs_{\sigma(m)} := gs_{\sigma(m),0^n} = x_{\sigma_1}^{d_1}\cdots x_{\sigma_n}^{d_n}.
\end{equation}
is the `classical' GS monomial indexed by $\sigma(m)$.
Our assumption on $m$ guarantees that $\Sigma$ lies in the span of $\BBB_{n,k}$ modulo $J_{n,k}$.

If $\mu(m)_1 > n-k$, then $e_{\mu(m)}(\xx_n) \equiv 0$ modulo $J_{n,k}$.  It follows that $m$ lies in the 
span of $\BBB_{n,k}$ modulo $J_{n,k}$, which is a contradiction.

If $\mu(m)_1 \leq n-k$, then by the definition of $\lambda(m)$, $d(m)$, and $\mu(m)$, we may write
\begin{equation}
m = gs_{\sigma(m)} \cdot x_{\sigma_1}^{\mu(m)'_1} \cdots x_{\sigma_{n-k}}^{\mu(m)'_{n-k}},
\end{equation}
where $\mu(m)'_1 \geq \cdots \geq \mu(m)'_{n-k} \geq 0$ is the conjugate of $\mu(m)$.  
Suppose $\mu(m)'_1 \geq k - \des(\sigma(m))$.  
Since the exponent of $x_{\sigma_1}$ in $gs_{\sigma(m)}$ equals $\des(\sigma(m))$,
we then have $x_{\sigma_1}^k \mid m$, which contradicts the assumption that $m$ has no variables with power
$\geq k$.  
Therefore, we have $\mu(m)'_1 < k - \des(\sigma(m))$.  This means that $m \in \GS_{n,k}$ and $m = gs_{w,\ii}$ for 
some pair $(w,\ii)$.
(In fact, we can take $(w,\ii)=(\sigma(m),\mu')$.)  However, our assumption on $\BBB_{n,k}$ guarantees that 
\begin{equation}
m = gs_{w,\ii} = b_{w,\ii} - \sum_{m' < m} c_{m'} \cdot m'
\end{equation}
for some scalars $c_{m'} \in \FF$.  
Then our assumption on $m$ implies that 
$m$ lies in the span of $\BBB_{n,k}$ modulo $J_{n,k}$, which is a contradiction.
\end{proof}

\begin{corollary}
\label{garsia-stanton-basis}
Let $k \leq n$ be positive integers.  The set $\GS_{n,k}$ of generalized Garsia-Stanton monomials
descends to a basis of $S_{n,k}$.
\end{corollary}


For example, suppose $(n,k) = (7,5)$ and $w = 6213745$.  Then $\des(w) = 2$ and the classical
GS monomial is
$gs_w = (x_6)(x_6 x_2 x_1 x_3 x_7)$.
We have $n-k = 2$ and $k - \des(w) = 3$, so that this classical GS monomial gives rise to the following six elements of
$\GS_{n,k}$:
\begin{center}
$\begin{array}{ccc}
(x_6)(x_6 x_2 x_1 x_3 x_7) & (x_6)(x_6 x_2 x_1 x_3 x_7)(x_6) & (x_6)(x_6 x_2 x_1 x_3 x_7)(x_6^2) \\
(x_6)(x_6 x_2 x_1 x_3 x_7) (x_6 x_2) & (x_6)(x_6 x_2 x_1 x_3 x_7) (x_6^2 x_2) &
(x_6)(x_6 x_2 x_1 x_3 x_7) (x_6^2 x_2^2).
\end{array}$
\end{center}

\section{Module structure over 0-Hecke algebra}
\label{Module}

In this section we prove that an isomorphism $S_{n,k} \cong \FF[\OP_{n,k}]$ of (ungraded) $H_n(0)$-modules.

\subsection{Ordered set partitions}
\label{osp}
We first describe the $H_n(0)$-module structure of $\FF[\OP_{n,k}]$.
Recall that if $\alpha \models n$ is a composition, then $P_{\alpha}$ is the corresponding 
indecomposable projective $H_n(0)$-module.
We need a family of projective $H_n(0)$-modules which are indexed by pairs of compositions related by refinement.
Let $\alpha, \beta \models n$ be two compositions satisfying $\alpha \preceq \beta$.
Let $P_{\alpha,\beta}$ be the $H_n(0)$-module given by
\begin{equation}\label{P-alpha-beta}
P_{\alpha,\beta} := H_n(0) \pib_{w_0(\alpha)} \pi_{w_0(\beta^c)}.
\end{equation}
In particular, we have $P_{\alpha,\alpha} = P_{\alpha}$.
More generally, we have the following structural result on $P_{\alpha,\beta}$.

\begin{lemma}
\label{alpha-beta-structure}
Let $\alpha, \beta \models n$ and assume $\alpha \preceq \beta$.  
Then $P_{\alpha,\beta}$ has basis
\begin{equation}
\{ \pib_w \pi_{w_0(\beta^c)} \,:\, w \in \symm_n, \, \Des(\alpha) \subseteq \Des(w) \subseteq \Des(\beta) \}.
\end{equation}
\label{ab-basis}
Let $\{ \gamma\models n: \alpha\preceq \gamma \preceq \beta \} = \{\gamma^1,\ldots,\gamma^r\}$ with $\gamma^i \preceq \gamma^j \Rightarrow i\le j$.
Then $P_{\alpha,\beta}$ is the direct sum of submodules $P_i\cong  P_{\gamma^i}$ for all $i\in[r]$, 
and each $P_i$ has basis $\{z_w: w\in\symm_n,\ \Des(w)=\Des(\gamma^i)\}$, where
\[ z_w = \pib_w \pi_{w_0(\beta^c)} + 
\sum_{ \substack{ u\in\symm_n \\ \Des(u) = \Des(\gamma^j),\ j>i }} c_u \pib_u\pi_{w_0(\beta^c)}, \quad c_u\in\FF.\]
\end{lemma}

\begin{proof}
We have a filtration $P_{\alpha,\beta} = M^{(1)} \supseteq\cdots\supseteq M^{(r)}\supseteq M^{(r+1)}=0$ of $H_n(0)$-modules,
where $M^{(i)}$ is the $\FF$-span of the disjoint union
\[ \bigsqcup_{j\ge i} \{ \pib_w\pi_{w_0(\beta^c)}: w\in\symm_n,\ \Des(w) = \Des(\gamma^j) \}\]
for $i=1,\ldots,r$.
The proof of \cite[Theorem 3.2]{HuangTab} implies $M^{(i)}/M^{(i+1)} \cong P_{\gamma^i}$ for all $i\in[r]$.
Since any short exact sequence of projective modules splits, we can write $P_{\alpha,\beta}$ as the direct sum of submodules $P_1,\ldots,P_r$, where each $P_i$ is isomorphic to $M^{(i)}/M^{(i+1)}$ and has a basis
\[ \{ \pib_w \pib_{w_0(\beta)} + \text{terms in $M^{(i+1)}$} : w\in\symm_n,\ \Des(w)=\Des(\gamma^i) \}.\]
The result follows.
\end{proof}

Also recall that, for each composition $\alpha=(\alpha_1,\ldots,\alpha_\ell) \models n$, we denote by $\OP_{\alpha}$ the collection of ordered set partitions of shape $\alpha$, i.e., pairs $(w,\alpha)$ for all $w\in\symm_n$ with $\Des(w)\subseteq \Des(\alpha)$.

\begin{lemma}
\label{osp-substructure}
Let $\alpha = (\alpha_1,\ldots,\alpha_\ell)$ be a composition of $n$. Then $\FF[\OP_\alpha]$ is a cyclic $H_n(0)$-module generated by the ordered set partition $(12\cdots n,\alpha)$ and is isomorphic to $P_{n,\alpha}$ via the map defined by sending $(w,\alpha)$ to $\pib_w \pi_{w_0(\alpha^c)}$ for all $w\in\symm_n$ with $\Des(w)\subseteq \Des(\alpha)$.
\end{lemma}

\begin{proof}

Huang~\cite[(3.3)]{HuangTab} defined an action of $H_n(0)$ on the $\FF$-span $P_{\alpha_1\oplus\cdots\oplus\alpha_\ell}$ of standard tableaux of skew shape $\alpha_1\oplus\cdots\oplus\alpha_\ell$, where $\alpha_1\oplus\cdots\alpha_\ell$ is a disconnected union of rows of lengths $\alpha_1,\ldots,\alpha_\ell$, ordered from southwest to northeast. 
There is an obvious isomorphism $\FF[\OP_\alpha] \cong P_{\alpha_1\oplus\cdots\oplus\alpha_\ell}$ by sending an ordered set partition $(B_1|\cdots|B_k)$ to the tableau whose rows are $B_1,\ldots,B_k$ from southwest to northeast. 
Combining this with the isomorphism $P_{\alpha_1\oplus\cdots\oplus\alpha_\ell} \cong P_{n,\alpha}$ provided by \cite[Theorem 3.3]{HuangTab} gives the desired result.
\end{proof}

\begin{proposition}
\label{osp-structure}
Let $k \leq n$ be positive integers. 
Then we have isomorphisms of $H_n(0)$-modules:
\begin{equation}
\FF[\OP_{n,k}] \cong \bigoplus_{\substack{ \alpha\models n \\ \ell(\alpha)=k }} \FF[\OP_\alpha] \cong \bigoplus_{\beta \models n} P_{\beta}^{\oplus{n - \ell(\beta) \choose k - \ell(\beta)}}.
\end{equation}
\end{proposition}

\begin{proof}
Since $\OP
_{n,k}$ is the disjoint union of $\OP_{n,\alpha}$ for all compositions $\alpha\models n$ of length $\ell(\alpha)=k$, the first desired isomorphism follows.
Applying Lemma~\ref{alpha-beta-structure} and Lemma~\ref{osp-substructure} to each $\OP_{n,\alpha}$ gives a direct sum decomposition of $\FF[\OP_{n,k}]$ into projective indecomposable modules.
The multiplicity of $P_\beta$ in this direct sum equals 
\[ |\{\beta\preceq\alpha: \ell(\alpha)=k\}| = { n-\ell(\beta) \choose k-\ell(\beta) }\]
for each $\beta\models n$. 
The second desired isomorphism follows.
\end{proof}

For example, we have $\FF[\OP_{13}] \cong P_{13} \oplus P_4$, $\FF[\OP_{22}]\cong P_{22} \oplus P_4$, 
$\FF[\OP_{31}] \cong P_{31}\oplus P_4$, and 
\begin{equation}\label{OP42}
 \FF[\OP_{4,2}] \cong P_{13}\oplus P_{22} \oplus P_{31} \oplus P_4^{\oplus 3}.
\end{equation}

\begin{figure}[h]
\[ \xymatrix{
1|234 \ar@(ur,rd)[]^{\pib_2 = \pib_3 = 0 } \ar[d]_{\pib_1} \\
2|134 \ar@(ur,rd)[]^{\substack{\pib_1 = -1 \\ \pib_3 = 0 }} \ar[d]_{\pib_2} \\
3|124 \ar@(ur,rd)[]^{\substack{\pib_1 = 0 \\ \pib_2 = -1 }} \ar[d]_{\pib_3} \\
4|123 \ar@(ur,rd)[]^{\substack{\pib_1 = \pib_2 = 0 \\ \pib_3 = -1 }} }
\quad \xymatrix{
& 12|34 \ar@(ur,rd)[]^{\pib_1 = \pib_3 = 0 } \ar[d]_{\pib_2} \\
& 13|24 \ar@(ur,r)[]^{\pib_2 = -1} \ar[ld]^{\pib_1} \ar[rd]_{\pib_3} \\
23|14 \ar@(ul,dl)[]_{\substack{\pib_1 = -1 \\ \pib_2 = 0 }} \ar[rd]^{\pib_3} &&
14|23 \ar@(ur,rd)[]^{\substack{\pib_2 = 0 \\ \pib_3 = -1 }} \ar[ld]_{\pib_1} \\
& 24|13 \ar@(r,rd)[]^{\pib_1 = \pib_3 = -1} \ar[d]_{\pib_2} \\
& 34|12 \ar@(ur,rd)[]^{ \substack{ \pib_1 = \pib_3 = 0 \\ \pib_2 = -1 } } } 
\quad \xymatrix{
123|4 \ar@(ur,rd)[]^{\pib_1 = \pib_2 = 0} \ar[d]_{\pib_3} \\
124|3 \ar@(ur,rd)[]^{\substack{\pib_1 = 0 \\ \pib_3 = -1 }} \ar[d]_{\pib_2} \\
134|2 \ar@(ur,rd)[]^{\substack{\pib_2 = -1 \\ \pib_3 = 0 }} \ar[d]_{\pib_1} \\
234|1 \ar@(ur,rd)[]^{\substack{\pib_1 = -1 \\ \pib_2 = \pib_3 = 0 }} 
} \]
\[ \OP _{13} \cong P_4 \oplus P_{13} \hspace{1.1in} \OP _{22} \cong P_4 \oplus P_{22} \hspace{1in} \OP _{31} \cong P_4 \oplus P_{31} \]
\caption{A decomposition of $\FF[\OP_{4,2}]$} \label{OP42}
\end{figure}

\subsection{0-Hecke action on polynomials}
Our next task is to show that $S_{n,k}$ has the same isomorphism type as the $H_n(0)$-module of 
Proposition~\ref{osp-structure}.
To do this, we will need to study the action of $H_n(0)$ on the polynomial ring $\FF[\xx_n]$ via the isobaric Demazure operators $\pi_i$ defined in \eqref{Demazure}.
Using the relation $\pib_i=\pi_i-1$, we have
\[ \pib_i(f) := \frac{x_{i+1} f - x_{i+1} (s_i(f))}{x_i - x_{i+1}},\quad \forall i\in[n-1],\ \forall f\in\FF[\xx_n]. \]
Thus for an arbitrary monomial $x_1^{a_1} \cdots x_n^{a_n}$, we have 
\begin{equation}\label{pi-bar}
\pib_i(x_1^{a_1} \cdots x_n^{a_n}) = \begin{cases}
\displaystyle (x_1^{a_1} \cdots x_{i-1}^{a_{i-1}} x_{i+2}^{a_{i+2}} \cdots x_n^{a_n}) \sum_{j = 1}^{a_i - a_{i+1}} x_i^{a_i - j} x_{i+1}^{a_{i+1}+j} & a_i \ge a_{i+1} \\
\displaystyle -(x_1^{a_1} \cdots x_{i-1}^{a_{i-1}} x_{i+2}^{a_{i+2}} \cdots x_n^{a_n})   \sum_{j = 0}^{a_{i+1} - a_i-1} x_i^{a_i + j} x_{i+1}^{a_i - j} & a_i<a_{i+1}.
\end{cases} 
\end{equation}
Using this we have the following triangularity result.

\begin{lemma}
\label{demazure-leading-term}
Let $\dd = (d_1 \geq \dots \geq d_n)$ be a weakly decreasing vector of nonnegative integers and let 
$\xx^{\dd} = x_1^{d_1} \cdots x_n^{d_n}$ be the corresponding monomial in $\FF[\xx_n]$.
Suppose $w \in \symm_n$ satisfies $\Des(w) \subseteq \Des(\dd)$.  
Then the leading term of the polynomial
$\pib_w (\xx^{\dd})$ under {\tt neglex} is the monomial $w(\xx^{\dd})$.
\end{lemma}

\begin{proof}
The proof is similar to~\cite[Lemma 4.1]{Huang}.
We induct on the length $\ell(w)$ of the permutation $w$.  If $\ell(w) = 0$, then $w$ is the identity permutation
and the lemma is trivial.  Otherwise, we may write $w = s_j v$, where $j\in[n-1]$ and $v \in \symm_n$ satisfies $\ell(w) = \ell(v) + 1$.
Wa have $j\in \Des(w^{-1})$, $j\notin \Des(v^{-1})$, and $\Des(v) \subseteq \Des(w)\subseteq \Des(\dd)$.
By induction we have
\begin{equation}
\pib_v(\xx^{\dd}) = v(\xx^{\dd}) + \sum_{m < \xx^{\dd}} c_m \cdot m
\end{equation}
for some scalars $c_m \in \FF$, where $<$ is {\tt neglex}.
Since $j\notin \Des(v^{-1})$, we have $v^{-1}(j) < v^{-1}(j+1)$ and thus $d_{v^{-1}(j)} \geq d_{v^{-1}(j+1)}$.
Since $wv^{-1}(j) = s_j(j) > s_j(j+1) = wv^{-1}(j+1)$, there exists an element of $[v^{-1}(j),v^{-1}(j+1)-1]$ which belongs to $\Des(w)\subseteq \Des(\dd)$.
This implies $d_{v^{-1}(j)} > d_{v^{-1}(j+1)}$.
Then by \eqref{pi-bar}, applying $\pib_j$ to $v(\xx^{\dd}) = x_{v(1)}^{d_1}\cdots x_{v(n)}^{d_n}$ yields a polynomial whose $\neglex$-leading term is $s_iv(\xx^\dd) = w(\xx^{\dd})$.  
On the other hand, \eqref{pi-bar} also implies that applying $\pib_i$ to any monomial which is $< v(\xx^{\dd})$ in $\neglex$ cannot yield monomials which are 
$\geq s_iv(\xx^\dd)=w(\xx^{\dd})$ in $\neglex$.
\end{proof}

We will decompose the quotient $S_{n,k}$ into a direct sum of projective modules of the form $P_{\alpha,\beta}$ defined in \eqref{P-alpha-beta}.
This decomposition will ultimately rest on the following lemma.

\begin{lemma}
\label{alpha-beta-monomial}
Let $\dd = (d_1 \geq \cdots \geq d_n)$ be a weakly decreasing sequence of nonnegative integers.
Suppose $\alpha,\beta \models n$ such that $\alpha\preceq\beta$ and $\Des(\dd) = \Des(\beta)$.
Then $H_n(0) \pib_{w_0(\alpha)} \xx^{\dd}$ has basis
\begin{equation}
\label{x-basis}
\left\{ \pib_w(\xx^{\dd}) \,:\, \Des(\alpha) \subseteq \Des(w) \subseteq \Des(\beta) \right\}.
\end{equation}
Furthermore, sending each element $\pib_w(\xx^\dd)$ in the basis \eqref{x-basis} to $\pib_w\pi_{w_0(\beta^c)}$ gives an isomorphism $H_n(0) \pib_{w_0(\alpha)} \xx^{\dd} \cong P_{\alpha,\beta}$ of $H_n(0)$-modules.
\end{lemma}

\begin{proof}
Let $1 \leq i \leq n-1$.  If $i \notin \Des(\beta)$, then the monomial $\xx^{\dd}$ is symmetric in $x_i$ and $x_{i+1}$,
so that $\pib_i(\xx^{\dd}) = 0$ by \eqref{pi-bar}.  
More generally, if $w \in \symm_n$ is such that $\Des(w) \not\subseteq \Des(\beta)$ then $\pib_w(\xx^{\dd}) = 0$ because there exists a reduced expression for $w$ ending in $s_i$ for some $i\in\Des(w)\setminus\Des(\beta)$.

By the last paragraph and the fact that $w_0(\alpha)$ is the left weak
Bruhat minimal permutation with descent set $\alpha$,
the module $H_n(0) \pib_{w_0(\alpha)} \xx^{\dd}$ is spanned by the set \eqref{x-basis}.
This set is linearly independent and hence a basis for $H_n(0) \pib_{w_0(\alpha)} \xx^{\dd}$, since by Lemma~\ref{demazure-leading-term} and the equality $\Des(\dd) = \Des(\beta)$, any two distinct elements $\pib_w(\xx^{\dd})$ and $\pib_{w'}(\xx^{\dd})$ of this set have $\neglex$ leading monomials $w(\xx^\dd)$ and $w'(\xx^\dd)$, which are distinct by $\Des(w)\subseteq \Des(\dd)$ and $\Des(w')\subseteq \Des(\dd)$.

By Lemma~\ref{alpha-beta-structure}, the module
$P_{\alpha,\beta}$ has basis given by \eqref{ab-basis}.
Thus the assignment $\pib_w(\xx^{\dd}) \mapsto \pib_w \pi_{w_0(\beta^c)}$ induces a linear isomorphism
from $H_n(0) \pib_{w_0(\alpha)} \xx^{\dd}$  to $P_{\alpha,\beta}$.  To check that this is an isomorphism of 
$H_n(0)$-modules, let $1 \leq i \leq n-1$.  We compare the action of $\pib_i$ on the bases 
(\ref{x-basis}) and (\ref{ab-basis}) as follows.  Let $w \in \symm_n$ satisfy 
$\Des(\alpha) \subseteq \Des(w) \subseteq \Des(\beta)$.

If $i \in \Des(w^{-1})$, then there is a reduced
expression for $w$ starting with $s_i$ and
 $\pib_i$ acts by the scalar $-1$ on both $\pib_w(\xx^{\dd})$
and $\pib_w \pi_{w_0(\beta^c)}$ since $\pib_i^2=-\pib_i$.

If $i \notin \Des(w^{-1})$ and $\Des(s_i w) \subseteq \Des(\beta)$, then 
the polynomial $\pib_i \pib_w (\xx^{\dd}) = \pib_{s_i w}(\xx^{\dd})$ lies in the basis (\ref{x-basis})
and the algebra element $\pib_i \pib_w \pi_{w_0(\beta^c)} = \pib_{s_i w} \pi_{w_0(\beta^c)}$
lies in the basis (\ref{ab-basis}).

If $i \notin \Des(w^{-1})$ and $\Des(s_i w) \not\subseteq \Des(\beta)$, we have 
$\pib_i \pib_w (\xx^{\dd}) = \pib_{s_i w} (\xx^{\dd}) = 0$ by the observation in the first paragraph.
On the other hand, we also have $\pib_i \pib_w \pi_{w_0(\beta^c)} = \pib_{s_iw}\pi_{w_0(\beta^c)} = 0$, since $s_iw$ has a reduced expression ending with $s_j$ for some $j\in\Des(\beta^c)$ and $\pib_j\pi_{w_0(\beta^c)}=0$ by the relation $\pib_j\pi_j=0$.
\end{proof}

\subsection{Decomposition of $S_{n,k}$}
We begin by introducing a family of $H_n(0)$-submodules of $S_{n,k}$.

\begin{defn}
Let $A_{n,k}$ be the set of all pairs $(\alpha, \ii)$, where $\alpha \models n$ 
is a composition whose first part satisfies $\alpha_1> n-k$
and $\ii = (i_1, \dots, i_n)$ is a sequence of nonnegative integers satisfying 
\begin{equation*}
k - \ell(\alpha) \geq i_1 \geq \cdots \geq i_{n-k} \geq 0 = i_{n-k+1} = \cdots = i_n.
\end{equation*}
Given a pair $(\alpha, \ii) \in A_{n,k}$, let $N_{\alpha,\ii}$ be the $H_n(0)$-module
generated by the image of the polynomial $\pib_{w_0(\alpha)}(\xx_{\alpha,\ii})$ in the quotient ring $S_{n,k}$.
\end{defn}

For example, let $(n,k) = (6,3)$.  Eliminating the $n-k = 3$ trailing zeros from the $\ii$ sequences, we have
\begin{equation*}
A_{6,3} = \left\{ 
\begin{array}{c}
(411, 000), (42, 111), (42, 110), (42, 100), (42, 000), (51, 111), (51, 110), (51, 100), (51,000), \\
(6, 222), (6,221), (6,220), (6,211), (6,210), (6,200), (6,111), (6,110), (6,100), (6,000)
\end{array}
 \right\}.
\end{equation*}

It will turn out that the $N_{\alpha,\ii}$ modules are special cases of the $P_{\alpha,\beta}$ modules.
Recall that, if $\alpha \models n$ and if $\ii$ is a length $n$ integer sequence, the composition
$\alpha \cup \ii \models n$ is characterized by $\Des(\alpha \cup \ii ) = \Des(\alpha) \cup \Des(\ii)$.
We will prove that if $(\alpha, \ii) \in A_{n,k}$, then 
$N_{\alpha,\ii} \cong P_{\alpha,\alpha \cup \ii}$.
To prove this fact, we will need a modification of the GS basis $\GS_{n,k}$ of $S_{n,k}$.
This modified basis will come from the following lemma, which 
states that the collection of GS monomials $\GS_{n,k}$
is related in a unitriangular way with the collection of polynomials
\begin{equation*}
\{ \pib_w(x_{\alpha,\ii}) \,:\, 
(\alpha,\ii) \in A_{n,k}, \, \Des(\alpha) \subseteq \Des(w) \subseteq \Des(\alpha \cup \ii) \}.
\end{equation*}

\begin{lemma}
\label{gs-n-lemma}  
Let $k \leq n$ be positive integers and endow monomials in $\FF[\xx_n]$
with $\neglex$ order.

\noindent(i)  Let $(\alpha, \ii) \in A_{n,k}$ and $w \in \symm_n$ such that 
$\Des(\alpha) \subseteq \Des(w) \subseteq \Des(\alpha \cup \ii)$.  
Then the leading term of 
$\pib_w(\xx_{\alpha,\ii})$ is $w(\xx_{\alpha,\ii}) = gs_{w,\ii'} \in \GS_{n,k}$, where
$\ii' = (i'_1, \dots, i'_n)$ is related to $\ii = (i_1, \dots, i_n)$ by
\begin{equation}\label{i2i'}
i'_j = i_j - | \{r \in \Des(w)\cap[n-k] \,:\, r \geq j \} |.
\end{equation}

\noindent(ii) Let $gs_{w,\ii'} \in \GS_{n,k}$ be a GS monomial.  
Then $gs_{w,\ii'}$ is the leading term of $\pib_w(\xx_{\alpha,\ii})$ 
for some $w\in\symm_n$ and some $(\alpha,\ii) \in A_{n,k}$ satisfying
$\Des(\alpha) \subseteq \Des(w) \subseteq \Des(\alpha \cup \ii)$ if and only if 
\begin{itemize}
\item  $\alpha \models n$ is characterized by $\Des(\alpha) = \Des(w) \setminus [n-k]$, and
\item the sequence $\ii = (i_1, \dots, i_n)$ is related to the sequence $\ii' = (i'_1, \dots, i'_n)$ by Equation~\eqref{i2i'}.
\end{itemize}
\end{lemma}

\begin{proof}
(1)  Since $\Des(w) \subseteq \Des(\alpha \cup \ii)$,
Lemma~\ref{demazure-leading-term} applies to show that the $\neglex$ leading term of 
$\pib_w(\xx_{\alpha,\ii})$ is $w(\xx_{\alpha,\ii})$.  We need to show that
$gs_{w,\ii'}$ is actually a GS basis element.

It is clear that $i'_j = i_j = 0$ for $j > n-k$.
We next claim that the sequence $\ii'$ is weakly decreasing.  To see this,  let $1 \leq j \leq n-k$ and
note that 
\begin{equation}
i'_j - i'_{j+1} = \begin{cases}
i_j - i_{j+1} - 1 & j \in \Des(w)\cap[n-k], \\
i_j - i_{j+1} & j \notin \Des(w)\cap[n-k].
\end{cases}
\end{equation}
Since $\ii$ is a weakly decreasing sequence and $i_j = i_{j+1}$ implies
$j \notin \Des(\alpha \cup \ii) \supseteq \Des(w)$, we conclude that $i'_j \geq i'_{j+1}$.
Finally, we have $\Des(w)\cap[n-k] = \Des(w)\setminus\Des(\alpha)$ since the definition of $A_{n,k}$ implies $D(\alpha)\cap[n-k]=\emptyset$.
Then
\begin{equation}
i_1' = i_1 - | \Des(w)\cap[n-k] | = i_1 - \des(w) + \ell(\alpha) -1 < k - \des(w),
\end{equation}
so that $gs_{w,\ii} \in \GS_{n,k}$ is a genuine GS basis element.

Next, we show $w(\xx_{\alpha,\ii}) = gs_{w,\ii'}$.
Let $1 \leq j \leq n$. Since $\Des(w)\cap[n-k] = \Des(w)\setminus\Des(\alpha)$, it follows from \eqref{i2i'} that
\begin{equation}
| \{r \in \Des(\alpha) \,:\, r \geq j \}| + i_j = | r \in \Des(w) \,:\, r \geq j \}| + i'_j.
\end{equation}
This means that the variable $x_{w(j)}$ has the same exponent in $w(\xx_{\alpha,\ii})$
as $gs_{w,\ii'}$.   We conclude that $w(\xx_{\alpha,\ii}) = gs_{w,\ii'}$.

(2)  Let $gs_{w,\ii'} \in \GS_{n,k}$.  Suppose $gs_{w,\ii'}$ is the $\neglex$ leading term of 
$\pib_w(\xx_{\alpha,\ii})$ for some $w\in\symm_n$ and some $(\alpha,\ii) \in A_{n,k}$ satisfying
$\Des(\alpha) \subseteq \Des(w) \subseteq \Des(\alpha \cup \ii)$. 

The definition of $A_{n,k}$ implies $D(\alpha)\cap[n-k]=\emptyset$ and $\Des(\ii)\subseteq[n-k]$.
Thus $\Des(\alpha)=\Des(w)\setminus[n-k]$ and $\Des(w)\setminus\Des(\alpha)=\Des(w)\cap[n-k]$.
Lemma~\ref{demazure-leading-term} guarantees that $gs_{w,\ii'} = w(x_{\alpha,\ii})$.   
Comparing the power of the variable $x_{w(j)}$ 
on both sides of this equality gives \eqref{i2i'} for all $1 \leq j \leq n$.

Conversely, given $gs_{w,\ii'} \in \GS_{n,k}$,
define $\alpha$ and $\ii$ as in the statement of the lemma.  We have $(\alpha,\ii) \in A_{n,k}$
and the $\neglex$ leading term of $\pib_w(\xx_{\alpha,\ii})$ is $gs_{w,\ii'}$ by 
similar arguments to those above.
\end{proof}

Lemma~\ref{gs-n-lemma} can be used to derive a new basis for the quotient  $S_{n,k}$.
This basis will be helpful in decomposing $S_{n,k}$ into a direct sum of
$H_n(0)$-modules of the form $N_{\alpha,\ii}$.

\begin{lemma}
\label{new-basis-lemma}
Let $k \leq n$ be positive integers.  The set of polynomials
\begin{equation}
\label{new-basis}
\{ \pib_w(\xx_{\alpha,\ii}) \,:\, (\alpha,\ii) \in A_{n,k}, \, w \in \symm_n, \, \Des(\alpha) \subseteq \Des(w) 
\subseteq \Des(\alpha \cup \ii) \}
\end{equation}
in $\FF[\xx_n]$ descends to a vector space basis of the quotient ring $S_{n,k}$.  Moreover, 
for any $(\alpha, \ii) \in A_{n,k}$ and any $w \in \symm_n$
with $\Des(\alpha) \subseteq \Des(w) \subseteq \Des(\alpha \cup \ii)$ we have
\begin{equation}
\label{degree-formula}
\deg( \pib_w(\xx_{\alpha,\ii}) ) = \deg( \xx_{\alpha,\ii}) = \maj(\alpha) + |\ii|.
\end{equation}
\end{lemma} 

\begin{proof}
By Lemma~\ref{gs-n-lemma}, the polynomials in the statement satisfy the conditions
of Lemma~\ref{garsia-stanton-type-bases}, and hence descend to a basis for $S_{n,k}$.
The degree formula is clear.
\end{proof}

In the coinvariant algebra
case $k = n$, the basis of Lemma~\ref{new-basis-lemma} appeared in \cite{Huang}.
As in \cite{Huang}, this modified GS-basis will facilitate analysis 
of the $H_n(0)$-structure of $S_{n,k}$.

\begin{theorem}
\label{s-decomposition-theorem}
Let $k \leq n$ be positive integers. 
For each $(\alpha,\ii) \in A_{n,k}$, the set of polynomials 
\begin{equation}
\label{new-basis-n}
\{ \pib_w(\xx_{\alpha,\ii}) \,:\, w \in \symm_n, \, \Des(\alpha) \subseteq \Des(w) 
\subseteq \Des(\alpha \cup \ii) \}
\end{equation}
descends to a basis for $N_{\alpha,\ii}$, and we have an isomorphism $N_{\alpha,\ii} \cong P_{\alpha,\alpha \cup \ii}$ of $H_n(0)$-modules by $\pib_w(\xx_{\alpha,\ii}) \mapsto \pib_w\pi_{w_0((\alpha\cup\ii)^c)}$.
Moreover, the $H_n(0)$-module $S_{n,k}$ satisfies
\begin{equation}
S_{n,k} = \bigoplus_{(\alpha,\ii) \in A_{n,k}} N_{\alpha,\ii} 
\cong \bigoplus_{\beta \models n} P_{\beta}^{\oplus{n - \ell(\beta) \choose k - \ell(\beta)}} \cong \FF[\OP_{n,k}].
\end{equation}
\end{theorem}

\begin{proof}
By Lemma~\ref{new-basis-lemma}, $S_{n,k}$ has a basis given by \eqref{new-basis}, which is the disjoint union of \eqref{new-basis-n} for all $(\alpha, \ii) \in A_{n,k}$.
Combining this with Lemma~\ref{alpha-beta-monomial}, we have the basis \eqref{new-basis-n} for $N_{\alpha,\ii}$ and the desired isomorphism $N_{\alpha,\ii} \cong P_{\alpha,\alpha\cup\ii}$ for all $(\alpha, \ii) \in A_{n,k}$.
The decomposition $S_{n,k} = \bigoplus_{(\alpha,\ii) \in A_{n,k}} N_{\alpha,\ii}$ follows.

Next, let $\beta\models n$ and count the multiplicity of $P_{\beta}$ as a direct summand in $S_{n,k}$.
Suppose $P_{\beta}$ is a direct summand of $N_{\alpha,\ii}$ for some $(\alpha,\ii) \in A_{n,k}$.
Since $\Des(\alpha \cup \ii)$ is the disjoint union $\Des(\alpha) \sqcup \Des(\ii)$ and
$\alpha \preceq \beta \preceq \alpha \cup \ii$, we must have $\Des(\alpha) = \Des(\beta) \setminus [n-k]$.
It follows that the multiplicity of $P_{\beta}$ in $S_{n,k}$ equals the number of 
choices of $\ii$ such that $(\alpha,\ii) \in A_{n,k}$ and
$\alpha \preceq \beta \preceq \alpha \cup \ii$, where $\alpha$ is characterized by
$\Des(\alpha) = \Des(\beta) \setminus [n-k]$.

We count the sequences $\ii = (i_1, \dots, i_n)$ of the above paragraph as follows.
Since $\Des(\beta) \cap [n-k] \subseteq \Des(\ii)$,  subtracting $1$ from $i_1, \dots, i_r$ for all
$r \in \Des(\beta) \cap [n-k]$ gives a weakly decreasing sequence $\ii' = (i'_1, \dots, i'_n)$ satisfying
$i'_{n-k+1} = \cdots = i'_n = 0$ and
\begin{equation*}
i'_1 \leq k - \ell(\alpha) - |\Des(\beta) \cap [n-k]| = k - \ell(\beta).
\end{equation*}
This gives a bijection from the collection of sequences $\ii$ of the last paragraph and 
sequences $\ii'$ satisfying the conditions of the last sentence.    
The number of such sequences $\ii'$ is ${n - \ell(\beta) \choose k - \ell(\beta)}$, which equals the multiplicity of $P_\beta$ in $S_{n,k}$.
Then Proposition~\ref{osp-structure} gives us $S_{n,k} \cong \FF[\OP_{n,k}]$, as desired.
\end{proof}

For example, let $(n,k) = (4,2)$.  We have 
$A_{4,2} = \{ (31, 00 ), ( 4, 11 ), ( 4, 10 ), ( 4, 00 ) \}$.
We get the corresponding $N_{\alpha,\ii}$ modules 
\begin{center}
$\begin{array}{cc}
N_{31,00} \cong P_{31,31} \cong P_{31} & N_{4,11} \cong P_{4,22} \cong P_4 \oplus P_{22} \\
N_{4,10} \cong P_{4,13} \cong P_4 \oplus P_{13} & 
N_{4,00} \cong P_{4,4} \cong P_4.
\end{array}$
\end{center}
Combining this with Theorem~\ref{s-decomposition-theorem}, 
we have $S_{4,2} \cong P_{22} \oplus P_{13} \oplus P_{31} \oplus P_4^{\oplus 3} \cong \FF[\OP_{4,2}]$. 
The following picture illustrates this isomorphism via the action of $H_4(0)$ on the basis \eqref{new-basis} of $S_{4,2}$ in
Lemma~\ref{new-basis-lemma}.
Note that the elements in this basis are polynomials in general, although they happen to be monomials in this example.

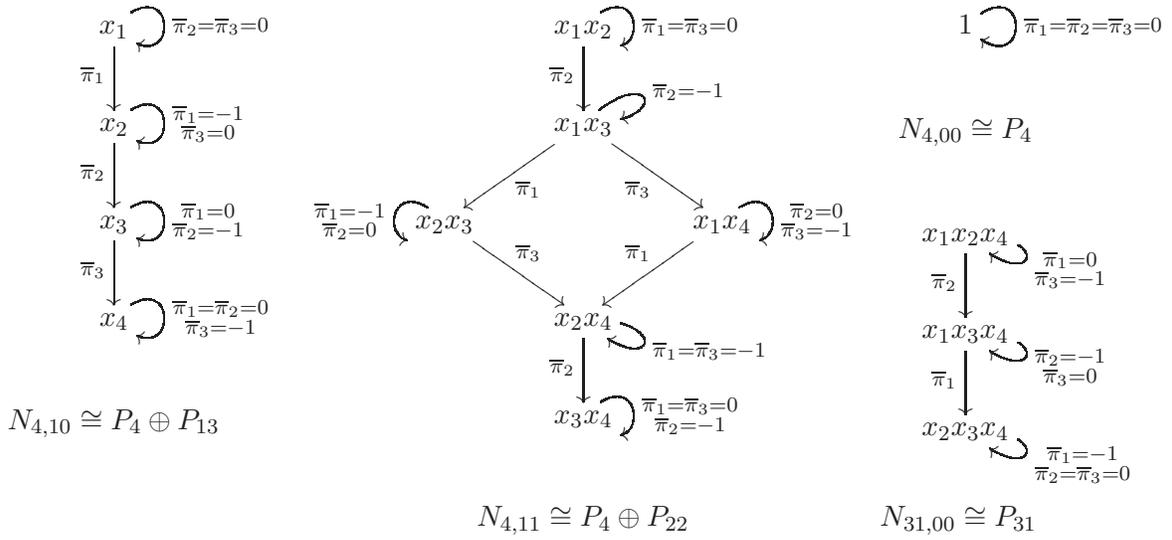
\begin{figure}[h]
\[ \xymatrix{
x_1 \ar@(ur,rd)[]^{\pib_2 = \pib_3 = 0 } \ar[d]_{\pib_1} \\
x_2 \ar@(ur,rd)[]^{\substack{\pib_1 = -1 \\ \pib_3 = 0 }} \ar[d]_{\pib_2} \\
x_3 \ar@(ur,rd)[]^{\substack{\pib_1 = 0 \\ \pib_2 = -1 }} \ar[d]_{\pib_3} \\
x_4 \ar@(ur,rd)[]^{\substack{\pib_1 = \pib_2 = 0 \\ \pib_3 = -1 }} \\ N_{4,10}  \cong P_4 \oplus P_{13} }
\quad \xymatrix{
& x_1x_2 \ar@(ur,rd)[]^{\pib_1 = \pib_3 = 0 } \ar[d]_{\pib_2} \\
& x_1x_3 \ar@(ur,r)[]^{\pib_2 = -1} \ar[ld]^{\pib_1} \ar[rd]_{\pib_3} \\
x_2x_3 \ar@(ul,dl)[]_{\substack{\pib_1 = -1 \\ \pib_2 = 0 }} \ar[rd]^{\pib_3} &&
x_1x_4 \ar@(ur,rd)[]^{\substack{\pib_2 = 0 \\ \pib_3 = -1 }} \ar[ld]_{\pib_1} \\
& x_2x_4 \ar@(r,rd)[]^{\pib_1 = \pib_3 = -1} \ar[d]_{\pib_2} \\
& x_3x_4 \ar@(ur,rd)[]^{ \substack{ \pib_1 = \pib_3 = 0 \\ \pib_2 = -1 } }   }
\quad \xymatrix{
1 \ar@(ur,rd)[]^{\pib_1 = \pib_2 = \pib_3 = 0 } \\  N_{4,00} \cong P_4 \\
x_1x_2x_4 \ar@(r,rd)[]^{\substack{\pib_1 = 0 \\ \pib_3 = -1 }} \ar[d]_{\pib_2} \\
x_1x_3x_4 \ar@(r,rd)[]^{\substack{\pib_2 = -1 \\ \pib_3 = 0 }} \ar[d]_{\pib_1} \\
x_2x_3x_4 \ar@(r,rd)[]^{\substack{\pib_1 = -1 \\ \pib_2 = \pib_3 = 0 }} 
} \]
\[   \hspace{1.8in}
N_{4,11}  \cong P_4 \oplus P_{22} \hspace{1in}
N_{31,00} \cong P_{31} \]
\caption{A decomposition of $S_{4,2}$} \label{S42}
\end{figure}

\section{Characteristic formulas}
\label{Characteristic}

In this section we derive formulas for the quasisymmetric and noncommutative symmetric characteristics
of the modules $S_{n,k}$.  To warm up, we calculate the degree-graded characteristics of the $N_{\alpha,\ii}$ modules.

\begin{lemma}
\label{n-characteristic-theorem}
Let $k \leq n$ be positive integers and let $(\alpha, \ii) \in A_{n,k}$.  The characteristics
$\ch_t(N_{\alpha,\ii})$ and $\Ch_{q,t}(N_{\alpha,\ii})$ have the following expressions:
\begin{align}
\ch_t(N_{\alpha,\ii}) &= t^{\maj(\alpha) + |\ii|} \sum_{\alpha \preceq \beta \preceq \alpha \cup \ii}  \sss_{\beta}, \\
\Ch_{q,t}(N_{\alpha,\ii}) &=  t^{\maj(\alpha) + |\ii|} \sum_{\substack
{w \in \symm_n \\ \Des(\alpha) \subseteq \Des(w) \subseteq \Des(\alpha \cup \ii)}}
q^{\inv(w) - \inv(w_0(\alpha))}  F_{\iDes(w)},
\end{align}
where in the second formula we view $N_{\alpha,\ii}$ as a cyclic module generated by $\pib_{w_0(\alpha)}(\xx_{\alpha,\ii})$.
\end{lemma}

\begin{proof}
As observed in the proof of Theorem~\ref{s-decomposition-theorem},
the set 
\begin{equation*}
\{ \pib_w(\xx_{\alpha,\ii}) \,:\, w \in \symm_n, \, \Des(\alpha) \subseteq \Des(w) 
\subseteq \Des(\alpha \cup \ii) \}
\end{equation*}
is a basis for $N_{\alpha,\ii}$.  Since the degree of the polynomial $\pib_w(\xx_{\alpha,\ii})$
is $\maj(\alpha) + |\ii|$, the formula for $\ch_t(N_{\alpha,\ii})$ follows from 
Theorem~\ref{s-decomposition-theorem}.  
For any $\ell\ge0$, the term $N_{\alpha,\ii}^{(\ell)}$ in the length filtration
of $N_{\alpha,\ii}^{(\ell)}$ has basis
\begin{equation*}
\{ \pib_w(\xx_{\alpha,\ii}) \,:\, w \in \symm_n, \, \Des(\alpha) \subseteq \Des(w) 
\subseteq \Des(\alpha \cup \ii), \, \ell(w) - \ell(w_0(\alpha)) \ge \ell \}.
\end{equation*}
The formula for $\Ch_{q,t}(N_{\alpha,\ii})$ follows.
\end{proof}

\begin{theorem}
\label{s-characteristic-theorem}
Let $k \leq n$ be positive integers.   We have
\begin{align}
\ch_t(S_{n,k}) &= \sum_{\alpha \models n} t^{\maj(\alpha)} {n - \ell(\alpha) \brack k - \ell(\alpha)}_t \sss_{\alpha}, \\
\Ch_{q,t}(S_{n,k}) &
= \sum_{(w,\alpha) \in \OP_{n,k}} q^{\inv(w)} t^{\maj(w,\alpha)} F_{\iDes(w)} \\
&= \sum_{w \in \symm_n} q^{\inv(w)} t^{\maj(w)} {n - \des(w) - 1 \brack k - \des(w) - 1}_t F_{\iDes(w)}.
\end{align}
\end{theorem}

\begin{proof}
Lemma~\ref{n-characteristic-theorem} implies that
\begin{equation}
\ch_t(S_{n,k}) = \sum_{(\alpha,\ii) \in A_{n,k}} t^{\maj(\alpha) + |\ii|} \sum_{\alpha \preceq \beta \preceq \alpha \cup \ii}
\sss_{\beta}.
\end{equation}
For a fixed composition $\alpha$, summing $t^{|\ii|}$ over all sequences $\ii = (i_1, \dots, i_n)$ with 
\begin{equation*}
k - \ell(\alpha) \geq i_1 \geq \cdots \geq i_{n-k} \geq 0 = i_{n-k+1} = \cdots = i_n 
\end{equation*}
generates a factor of 
${n - \ell(\alpha) \brack k - \ell(\alpha)}_t$.
This implies the formula for $\ch_t(S_{n,k})$.

In order to state the length-degree-bigraded quasisymmetric characteristic $\Ch_{q,t}(S_{n,k})$ of $S_{n,k}$,
we need a distinguished direct sum decomposition of $S_{n,k}$ into cyclic 0-Hecke modules.
Using the decomposition of $S_{n,k}$ into the direct sum of the cyclic modules $N_{\alpha,\ii}$ for all $(\alpha,\ii)\in A_{n,k}$ provided by Theorem~\ref{s-decomposition-theorem} would give a calibration term `$-\inv(w_0(\alpha))$' in the power of $q$, as in $\Ch_{q,t}(N_{\alpha,\ii})$.
To avoid this issue, we use another decomposition of $S_{n,k}$ from the $H_n(0)$-module isomorphisms
\[ S_{n,k} \cong \bigoplus_{\beta \models n} P_{\beta}^{\oplus{n - \ell(\beta) \choose k - \ell(\beta)}} 
\cong \FF[\OP_{n,k}] \cong\bigoplus_{ \substack{\alpha\models n \\ \ell(\alpha)=k }} \FF[\OP_\alpha] \]
provided by Theorem~\ref{s-decomposition-theorem} and Proposition~\ref{osp-structure}.
Fix $\beta\models n$ and a total order $\le$ for subsets of $[n-1]$ that is compatible with the partial order $\subseteq$.

On one hand, Proposition~\ref{osp-structure} implies that a direct summand of $\FF[\OP_{n,k}]$ isomorphic to $P_\beta$ is contained in $\FF[\OP_\alpha]$ for some $\alpha\models n$ satisfying $\ell(\alpha)=k$ and $\beta\preceq\alpha$. 
There are ${n-\ell(\beta)\choose k-\ell(\beta)}$ many choices of $\alpha$, each giving a direct summand of $\FF[\OP_{n,k}]$ isomorphic to $P_\beta$.
Moreover, Lemma~\ref{alpha-beta-structure} implies that this direct summand has a basis $\{\sigma_{w,\alpha}: w\in\symm_n,\ \Des(w) = \Des(\beta) \}$, where 
\[ \sigma_{w,\alpha} = (w,\alpha) + \sum_{ u\in\symm_n : \Des(u) > \Des(w)} a_u(u,\alpha), \quad a_u\in\FF.\]

On the other hand, by the proof of Theorem~\ref{s-decomposition-theorem}, a direct summand of $S_{n,k}$ isomorphic to $P_\beta$ is contained in $N_{\gamma,\ii}$, where $\gamma\models n$ is determined by $\Des(\gamma) = \Des(\beta)\setminus[n-k]$ and the sequence $\ii=(i_1,\ldots,i_n)$ is obtained from another sequence $\ii' = (i'_1, \dots, i'_n)$ by
\[ i_j := | \{ r \in \Des(\beta) \cap [n-k] : r\ge j | + i'_j, \quad \forall j\in[n]. \]
There are ${n-\ell(\beta)\choose k-\ell(\beta)}$ many choices for the sequence $\ii'$, since $\ii'$ satisfies
\[ k-\ell(\beta) \ge i'_1\ge \cdots \ge i'_{n-k}\ge i'_{n-k+1} = \cdots = i'_n=0.\]
Each choice of $\ii'$ gives a direct summand of $S_{n,k}$ isomorphic to $P_\alpha$. 
Lemma~\ref{alpha-beta-structure} implies that this direct summand has a basis $\{ \tau_{w,\gamma,\ii}: w\in \symm_n,\ \Des(w) = \Des(\beta) \}$, where
\[ \tau_{w,\gamma,\ii} := \pib_w(\xx_{\gamma,\ii}) + \sum_{u\in\symm_n: \Des(u)>\Des(w)} b_u \pib_u(\xx_{\gamma,\ii}).\]

Now we define a bijection between the choices of $\alpha$ to the choices of $\ii'$ by setting $i'_j$ be the number of elements of $\Des(\alpha)\setminus\Des(\beta)$ no less than the $j$th smallest element of $\Des(\alpha^c)$ for all $j\in[n-k]$ and setting $i'_{n-k+1} = \cdots = i'_n =0$.
This gives an explicit isomorphism between $\FF[\OP_{n,k}]$ and $S_{n,k}$, since we know a bijection between their direct summands isomorphic to $P_\beta$ for each $\beta\models n$.
For each $w\in\symm_n$ with $\Des(w)=\Des(\beta)$, the leading term of $\sigma_{w,\alpha}$ is $(w,\alpha)$, which is contained in the module at position $\inv(w)$ in the length filtration of $\FF[\OP_{\alpha}]$, and the degree of $\tau_{w,\gamma,\ii}$ equals $\maj(\gamma\cup\ii) = \maj(\beta)+|\ii'| = \maj(w,\alpha)$.
The first desired expression of $\Ch_{q,t}(S_{n,k})$ follows.

Next, by Lemma~\ref{gs-n-lemma}, for each fixed $w\in\symm_n$, $\pib_w(\xx_{\gamma,\ii})$ belongs to the basis of $S_{n,k}$ provided by Lemma~\ref{new-basis-lemma} if and only if $\Des(\gamma) = \Des(w)\setminus[n-k]$ and the sequence $\ii=(i_1,\ldots,i_n)$ is given by
\[ i_j := | \{ r \in \Des(w) \cap [n-k] : r\ge j | + i'_j, \quad \forall j\in[n] \]
where $\ii' = (i'_1, \dots, i'_n)$ is any sequence of nonnegative integers satisfying
\[ k - \Des(w)-1 \ge i'_1 \ge\cdots\ge i'_{n-k} \ge i'_{n-k+1} = \cdots = i'_n = 0. \]
The degree of $\pib_w(\xx_{\gamma,\ii})$ equals $\maj(\gamma\cup\ii) = \maj(w) + |\ii'|$, and summing $t^{|\ii'|}$ over all such sequences $\ii'$ gives ${n-\Des(w)-1 \brack k-\Des(w)-1}_t$.
The second expression of $\Ch_{q,t}(S_{n,k})$ follows.
\end{proof}


The first expression for $\Ch_{q,t}(S_{n,k})$ presented in Theorem~\ref{s-characteristic-theorem} is related
to an extension of the biMahonian distribution to ordered set partitions.  More precisely, let 
$\sigma \in \OP_{n,k}$ be an ordered set partition and represent $\sigma$ as $(w, \alpha)$, where
$w \in \symm_n$ is a permutation which satisfies $\Des(w) \subseteq \Des(\alpha)$.
We define the {\em length} statistic $\ell(\sigma)$ by
\begin{equation}
\ell(\sigma) = \ell(w,\alpha) := \inv(w).
\end{equation}
In the language of Coxeter groups, the permutation $w$ is the Bruhat minimal representative of the parabolic
coset $w \symm_{\alpha} = w (\symm_{\alpha_1} \times \cdots \times \symm_{\alpha_k})$, so that $\ell(\sigma)$
is the Coxeter length of this minimal element.

We have
\begin{equation}
\label{length-equation}
\sum_{\sigma \in \OP_{\alpha}} q^{\ell(\sigma)} = {n \brack \alpha_1, \dots, \alpha_k}_q.
\end{equation}
Summing Equation~\ref{length-equation} over all $\alpha \models n$ with $\ell(\alpha) = k$
gives a {\bf different} distribution than the generating function of $\maj$:
\begin{equation}
\label{maj-equation}
\sum_{\sigma \in \OP_{n,k}} q^{\maj(\sigma)} = \rev_q([k]!_q \cdot \Stir_q(n,k)),
\end{equation}
although these distributions both equal $[n]!_q$ in the case $k = n$.
\footnote{There is a different extension of the inversion/length statistic
on $\symm_n$ to $\OP_{n,k}$   \cite{RW, WMultiset, Rhoades, HRW, HRS} whose distribution is
$[k]!_q \cdot \Stir_q(n,k)$.}

By Theorem~\ref{s-characteristic-theorem} we have
\begin{equation}
\Ch_{q,t}(S_{n,k}) = \sum_{\sigma \in \OP_{n,k}} 
q^{\ell(\sigma)} t^{\maj(\sigma)} F_{\iDes(\sigma)},
\end{equation}
where $F_{\iDes(\sigma)} := F_{\iDes(w)}$ for $\sigma = (w, \alpha)$.
In other words, we have that $\Ch_{q,t}(S_{n,k})$ is the generating function
for the `biMahonian pair' $(\ell, \maj)$ on $\OP_{n,k}$ with quasisymmetric function weight $F_{\iDes(\sigma)}$.

We may also derive expressions for the degree-graded quasisymmetric characteristic $\Ch_t(S_{n,k})$.
It turns out that this quasisymmetric characteristic is actually a symmetric function by since $S_{n,k}$ is projective and $\Ch(P_\alpha) = s_\alpha \in\Sym$ as given in \eqref{ChP}.
We give an explicit expansion of $\Ch_t(S_{n,k})$ in the Schur basis.

\begin{corollary}
\label{s-characteristic-corollary}
Let $k \leq n$ be positive integers.  We have 
\begin{align}
\Ch_t(S_{n,k}) &= \sum_{(w,\alpha) \in \OP_{n,k}} t^{\maj(w,\alpha)} F_{\iDes(w)}  \\
&= \sum_{w \in \symm_n}  t^{\maj(w)} {n - \des(w) - 1 \brack k - \des(w) - 1}_t F_{\iDes(w)}  \\
&= \sum_{\alpha \models n} t^{\maj(\alpha)} {n - \ell(\alpha) \brack k - \ell(\alpha)}_t s_{\alpha}.
\end{align}
Moreover, the above symmetric function has expansion in the Schur basis given by
\begin{equation}
\label{schur-expansion-formula}
\Ch_t(S_{n,k}) = \sum_{Q \in \SYT(n)} t^{\maj(Q)} {n  - \des(Q) - 1 \brack k - \des(Q) - 1}_t s_{\shape(Q)}.
\end{equation}
\end{corollary}

\begin{proof}
The first and second expressions for $\Ch_t(S_{n,k})$ follow from Theorem~\ref{s-characteristic-theorem}
by setting $q = 1$ in the expressions for $\Ch_{q,t}(S_{n,k})$ given there.  The third expression for 
$\Ch_t(S_{n,k})$ follows from replacing $\sss_{\alpha}$ by $s_{\alpha}$ in $\ch_t(S_{n,k})$.

To derive Equation~\ref{schur-expansion-formula}, we start with 
$\Ch_t(S_{n,k}) = \sum_{w \in \symm_n}  t^{\maj(w)} {n - \des(w) - 1 \brack k - \des(w) - 1}_t F_{\iDes(w)}$
and apply the Schensted correspondence.  More precisely, the (row insertion)
Schensted correspondence gives a bijection
$w \mapsto (P(w), Q(w))$ from the symmetric group $\symm_n$ to ordered pairs of standard Young tableaux
with $n$ boxes having the same shape.  
For example, we have
\begin{equation*}
25714683 \, \,  \mapsto \, \, \, \,
\begin{footnotesize}
\begin{Young}
 1 & 3 & 6 & 8 \cr
 2 & 4 & 7 \cr
 5 
\end{Young} \,, \, \,
\begin{Young}
1 & 2 & 3 & 7 \cr
4 & 5 & 6 \cr
8 
\end{Young} \, 
\end{footnotesize}.
\end{equation*}

A {\em descent} of a standard tableau $P$ is a letter $i$ which appears in a row above 
the row containing $i+1$ in $P$.
We let $\Des(P)$ denote the set of descents of $P$, and define the corresponding descent number
$\des(P) := |\Des(P)|$ and major index $\maj(P) := \sum_{i \in \Des(P)} i$.
Under the Schensted  bijection we have
$\Des(w) = \Des(Q(w))$, so that  $\des(w) = \des(Q(w))$ and $\maj(w) = \maj(Q(w))$.
Moreover, we have $w^{-1} \mapsto (Q(w), P(w))$, so that $\iDes(w) = \Des(P(w))$.

Applying the Schensted correspondence, we see that
\begin{align}
\Ch_t(S_{n,k}) &= \sum_{w \in \symm_n}  t^{\maj(w)} {n - \des(w) - 1 \brack k - \des(w) - 1}_t F_{\iDes(w)} \\
&= \sum_{(P,Q)} t^{\maj(Q)} {n  - \des(Q) - 1 \brack k - \des(Q) - 1}_t F_{\Des(P)},
\end{align}
where the second sum is over all pairs $(P,Q)$ of standard Young tableaux with $n$ boxes satisfying 
$\shape(P) = \shape(Q)$.
Gessel \cite{Gessel} proved that for any $\lambda \vdash n$,
\begin{equation}
\label{gessel-result}
\sum_{P \in \SYT(\lambda)} F_{\Des(P)} = s_{\lambda},
\end{equation}
where the sum is over all standard tableaux $P$ of shape $\lambda$.
Applying Equation~\eqref{gessel-result} gives
\begin{align}
\sum_{(P,Q)} t^{\maj(Q)} {n  - \des(Q) - 1 \brack k - \des(Q) - 1}_t F_{\Des(P)} 
&= \sum_{Q} t^{\maj(Q)} {n  - \des(Q) - 1 \brack k - \des(Q) - 1}_t \sum_{P \in \SYT(\shape(Q))} F_{\Des(P)} \\
&= \sum_Q t^{\maj(Q)} {n  - \des(Q) - 1 \brack k - \des(Q) - 1}_t s_{\shape(Q)},
\end{align}
as desired.
\end{proof}

The Schur expansion of $\Ch_t(S_{n,k})$ given in Corollary~\ref{s-characteristic-corollary}
coincides (after setting $q = t$) with Schur expansion \cite[Cor. 6.13]{HRS} of the Frobenius image
of the graded $\symm_n$-module $R_{n,k}$.  That is, we have
\begin{equation}
\label{rs-coincidence}
\Ch_t(S_{n,k}) = \grFrob(R_{n,k};t).
\end{equation}

\section{Conclusion}
\label{Conclusion}

\subsection{Macdonald polynomials and Delta conjecture}
Equation~\ref{rs-coincidence} gives a connection between our work and the theory of Macdonald polynomials.
More precisely, the {\em Delta Conjecture} of Haglund, Remmel, and Wilson \cite{HRW} predicts that 
\begin{equation}
\Delta_{e_{k-1}}' e_n = \Rise_{n,k-1}(\xx;q,t) = \Val_{n,k-1}(\xx;q,t),
\end{equation}
where $\Delta_{e_{k-1}}'$ is the Macdonald eigenoperator defined by
\begin{equation}
\Delta_{e_{k-1}}': \tilde{H}_{\mu} \mapsto e_{k-1}[B_{\mu}(q,t) - 1] \cdot \tilde{H}_{\mu}
\end{equation}
and $\Rise_{n,k-1}(\xx;q,t)$ and $\Val_{n,k-1}(\xx;q,t)$ are certain combinatorially defined
quasisymmetric functions; see \cite{HRW} for  definitions.
By the work of 
Wilson \cite{WMultiset} and Rhoades \cite{Rhoades}, we have the following consequence of the Delta Conjecture:
\begin{equation}
\label{four-ways}
\Rise_{n,k-1}(\xx;q,0) = \Rise_{n,k-1}(\xx;0,q) = \Val_{n,k-1}(\xx;q,0) = \Val_{n,k-1}(\xx;0,q).
\end{equation}
If we let $C_{n,k}(\xx;q)$ denote the common symmetric function in Equation~\ref{four-ways},
the work of Haglund, Rhoades, and Shimozono \cite[Thm. 6.11]{HRS} implies that 
\begin{equation}
\grFrob(R_{n,k}; q) = (\rev_q \circ \omega) C_{n,k}(\xx;q),
\end{equation}
where $\omega$ is the standard involution on $\Sym$ sending $h_d$ to $e_d$ for all $d \geq 0$.
Equation~\ref{rs-coincidence} implies that
\begin{equation}
\Ch_t(S_{n,k}) = (\rev_t \circ \omega) C_{n,k}(\xx;t).
\end{equation}

The derivation of $\grFrob(R_{n,k};q)$ in \cite{HRS} has a different flavor from our derivation of 
$\Ch_t(S_{n,k})$; the definition of the rings $R_{n,k}$ is extended to include family 
$R_{n,k,s}$ involving a third parameter $s$.  The $R_{n,k,s}$ rings are related to the image of the $R_{n,k}$
rings under a certain idempotent in the symmetric group algebra $\QQ[\symm_n]$; this relationship forms 
the basis of an inductive derivation of $\grFrob(R_{n,k};q)$.
The coincidence of $\Ch_t(S_{n,k})$ and $\grFrob(R_{n,k};t)$ is mysterious to the authors.

\begin{problem}
Find a conceptual explanation of the identity $\Ch_t(S_{n,k}) = \grFrob(R_{n,k};t)$.
\end{problem}

\subsection{Tanisaki ideals}
Given a partition $\lambda \vdash n$, let $I_{\lambda} \subseteq \FF[\xx_n]$ denote the corresponding 
{\em Tanisaki ideal} (see \cite{GP} for a generating set of $I_{\lambda}$).
When $\FF = \QQ$, the quotient $R_{\lambda} := {\FF[\xx_n]} / {I_{\lambda}}$ is isomorphic to the 
cohomology ring of the Springer fiber attached to $\lambda$.
The quotient $R_{\lambda}$ is a graded $\symm_n$-module.
It is well known \cite{GP} that $\grFrob(R_{\lambda};q) = Q'_{\lambda}(\xx;q)$, where
$Q'_{\lambda}(\xx;q)$ is the dual Hall-Littlewood polynomial indexed by $\lambda$.

Huang proved that $I_{\lambda}$ is closed under the action of $H_n(0)$
on $\FF[\xx_n]$ if and only if $\lambda$ is a hook, 
so that the quotient $R_{\lambda}$ has the structure of a graded 0-Hecke module for hook shapes $\lambda$
\cite[Prop. 8.2]{Huang}.
Moreover, when $\lambda \vdash n$ is a hook, \cite[Cor. 8.4]{Huang} implies that 
$\Ch_t(R_{\lambda}) = \grFrob(R_{\lambda}; t) = Q'_{\lambda}(\xx;q)$.
When $\lambda \vdash n$ is not a hook, the quotient $R_{\lambda}$ does not inherit a 
0-Hecke action.  

In this paper, we modified the ideal $I_{n,k}$ of \cite{HRS} to obtain a new ideal $J_{n,k} \subseteq \FF[\xx_n]$
which is stable under the action of $H_n(0)$ on $\FF[\xx_n]$.  Moreover, we have
$\Ch_t \left( {\FF[\xx_n]} / {J_{n,k}} \right) = \grFrob \left( {\QQ[\xx_n]} / {I_{n,k}}; t \right)$.
This suggests the following problem.

\begin{problem}
\label{tanisaki-problem}
Let $\lambda \vdash n$.  Define a homogeneous ideal $J_{\lambda} \subseteq \FF[\xx_n]$ which is stable under the 
0-Hecke action on $\FF[\xx_n]$ such that
\begin{equation}
\Ch_t \left( {\FF[\xx_n]} / {J_{\lambda}} \right) = \grFrob(R_{\lambda};t) = Q'_{\lambda}(\xx;t).
\end{equation}
\end{problem}

When $\lambda$ is a hook, the Tanisaki ideal $I_{\lambda}$ is a solution to Problem~\ref{tanisaki-problem}.

\subsection{Generalization to reflection groups}
Let $W$ be a Weyl group.  There is an action of the 0-Hecke algebra $H_W(0)$ attached to $W$
on the Laurent ring of the weight lattice $Q$ of $W$.  If $W$ has rank $r$, this Laurent ring is isomorphic
to $\FF[x_1, \dots, x_r, x_1^{-1}, \dots, x_r^{-1}]$.  Huang described the 0-Hecke structure of the corresponding
coinvariant algebra \cite[Thm. 5.3]{Huang}.
On the other hand, Chan and Rhoades \cite{CR} described a generalization of the ideal $I_{n,k}$
of \cite{HRS} for the complex reflection groups $G(r,1,n) \cong \ZZ_r \wr \symm_n$.
It would be interesting to give an analog of the work in this paper for a wider class of reflection groups.

\section{Acknowledgements}
\label{Acknowledgements}

The authors are grateful to Jim Haglund, Adriano Garsia, Jeff Remmel, Mark Shimozono, and Andy Wilson
for helpful conversations.
B. Rhoades was partially supported by NSF Grant DMS-1500838.


\begin{thebibliography}{99}
 
 \bibitem{ABR} R. Adin, F. Brenti, and Y. Roichman.  Descent representations and multivariate statisitcs.
 {\it Trans. Amer. Math. Soc.}, {\bf 357} (2005), 3051--3082.
 
 
 \bibitem{Artin}  E. Artin.  {\it Galois Theory,} Second edition.
 Notre Dame Math Lectures, no. 2.  Notre Dame: University of Notre Dame, 1944.
 
 \bibitem{ASS} I. Assem, D. Simson\ and\ A. Skowro\'nski, {\it Elements of the representation theory of associative algebras. Vol. 1}, London Mathematical Society Student Texts, 65, Cambridge Univ. Press, Cambridge, 2006.
 
 \bibitem{BBSSZ}  C. Berg, N. Bergeron, F. Saliola, L. Serrano, and M. Zabrocki.
 Indecomposable modules for the dual immaculate basis of quasi-symmetric functions.
 {\it Proc. Amer. Math. Soc.}, {\bf 143} (2015), 991--1000.
 
 
 \bibitem{Bergeron}  F. Bergeron.  {\it Algebraic Combinatorics and Coinvariant Spaces.}
 CMS Treatises in Mathematics.
Boca Raton:  Taylor and Francis, 2009.

 
\bibitem{BjornerWachs}
A. Bj\"orner\ and\ M. L. Wachs, {\it Generalized quotients in Coxeter groups}, Trans. Amer. Math. Soc. {\bf 308} (1988), no.~1, 1--37.

 \bibitem{CR}  J. Chan and B. Rhoades.  Generalized coinvariant algebras for wreath products.
 Preprint, 2017.
 {\tt arXiv:1701.06256}.
 
 \bibitem{C}  C. Chevalley.  Invariants of finite groups generated by reflections.
 {\it Amer. J. Math.}, {\bf 77 (4)} (1955), 778--782.
 
 \bibitem{CLO} D. Cox, J. Little\ and\ D. O'Shea, {\it Ideals, varieties, and algorithms.}  Third Edition. 
 Undergraduate Texts in  Mathematics, Springer, New York, 1992.  
 
 

 \bibitem{Garsia}  A. M. Garsia.  Combinatorial methods in the theory of Cohen-Macaulay rings.
 {\it Adv. Math.}, {\bf 38} (1980), 229--266.
 
 \bibitem{GP}  A. M. Garsia and C. Procesi.  On certain graded $S_n$-modules and the $q$-Kostka
 polynomials.  {\it Adv. Math.}, {\bf 94 (1)} (1992), 82--138.
 
 \bibitem{GS}  A. M. Garsia and D. Stanton.  Group actions on Stanley-Reisner rings and invariants of permutation
 groups.  {\it Adv. Math.}, {\bf 51 (2)} (1984), 107--201.
 
 \bibitem{Gessel}  I. Gessel.  {\it Multipartite P-partitions and inner products of skew Schur functions},
 Combinatorics and algebra (Boulder, Colo., 1983), Contemp. Math., vol 34, Amer. Math. Soc.,
 Providence, RI, 1984, pp. 289--317.
  
 \bibitem{GrinbergReiner} D. Grinberg and V. Reiner, {\it Hopf Algebras in Combinatorics}, arXiv:1409.8356v4.
 
 
 
 

\bibitem{HRW}  J. Haglund, J. Remmel, and A. T. Wilson.  The Delta Conjecture.  
Accepted, {\it Trans. Amer. Math. Soc.}, 2016.  {\tt arXiv:1509.07058}.

\bibitem{HRS}  J. Haglund, B. Rhoades, and M. Shimozono.  Ordered set partitions,
generalized coinvariant algebras, and the Delta conjecture.  Preprint, 2016.



\bibitem{Huang}  J. Huang.  $0$-Hecke actions on coinvariants and flags.
{\it J. Algebraic Combin.}, {\bf 40} (2014), 245--278.

\bibitem{HuangTab}  J. Huang.  A tableau approach to the representation theory of 0-Hecke algebras.
{\it Ann. Comb.}, to appear, 2016.


\bibitem{KT}  D. Krob and J.-Y. Thibon.  Noncommutative symmetric functions IV:  Quantum linear
groups and Hecke algebras at $q = 0$.   {\it J. Algebraic Combin.}, {\bf 6} (1997), 339--376.


\bibitem{MacMahon}  P. A. MacMahon.  {\it Combinatory Analysis}, volume 1.  Cambridge University Press,
1915.


\bibitem{Norton}  P. N. Norton. 0-Hecke algebras. {\it J. Austral. Math. Soc. A}, {\bf 27} (1979), 337--357.


\bibitem{RW}  J. Remmel and A. T. Wilson.  An extension of MacMahon's Equidistribution
Theorem to ordered set partitions.
{\it J. Combin. Theory Ser. A}, {\bf 134} (2015), 242--277.

\bibitem{Rhoades}  B. Rhoades.  Ordered set partition statistics and the Delta Conjecture.
Preprint, 2016. {\tt arXiv:1605.04007}.

\bibitem{ST}  G. C. Shephard and J. A. Todd.  Finite unitary reflection groups.  {\it Can. J. Math.},
{\bf 6} (1954), 274--304.



\bibitem{Stanley}  R. P. Stanley.  Invariants of finite groups and their applications to combinatorics.
{\it Bull. Amer. Math. Soc.}, {\bf 1} (1979), 475--511.





\bibitem{VT}  S. van Willigenburg and V. Tewari.  Modules of the 0-Hecke algebra and 
quasisymmetric Schur functions.
{\it Adv. Math.}, {\bf 285} (2015), 1025--1065.


\bibitem{WMultiset}  A. T. Wilson.  An extension of MacMahon's Equidistribution Theorem
to ordered multiset partitions. 
{\it Electron. J. Combin.}, {\bf 23 (1)} (2016), P1.5.


  
\end{thebibliography}
\end{document}